\documentclass[sn-mathphys-num]{sn-jnl}% Math and Physical Sciences Numbered Reference Style 
\usepackage{graphicx}%
\usepackage{multirow}%
\usepackage{amsmath,amssymb,amsfonts}%
\usepackage{amsthm}%
\usepackage{mathrsfs}%
\usepackage[title]{appendix}%
\usepackage{xcolor}%
\usepackage{textcomp}%
\usepackage{manyfoot}%
\usepackage{booktabs}%
\usepackage{algorithm}%
\usepackage{listings}%
\usepackage{algorithmic}
\usepackage{float}
\usepackage{textcomp}
\RequirePackage{subcaption}
\ifpdf
  \DeclareGraphicsExtensions{.eps,.pdf,.png,.jpg,.}
\else
  \DeclareGraphicsExtensions{.eps}
\fi
\theoremstyle{thmstyleone}%
\newtheorem{theorem}{Theorem}%  meant for continuous numbers
\newtheorem{corollary}{Corollary}
\newtheorem{lemma}{Lemma}
\theoremstyle{thmstyletwo}%

\theoremstyle{thmstylethree}%

\begin{document}

\title[Article Title]{Relating Checkpoint Update Probabilities to Momentum Parameters in Single-Loop Variance Reduction Methods}

%%=============================================================%%
%% GivenName	-> \fnm{Joergen W.}
%% Particle	-> \spfx{van der} -> surname prefix
%% FamilyName	-> \sur{Ploeg}
%% Suffix	-> \sfx{IV}
%% \author*[1,2]{\fnm{Joergen W.} \spfx{van der} \sur{Ploeg} 
%%  \sfx{IV}}\email{iauthor@gmail.com}
%%=============================================================%%

\author[1]{\fnm{Hai} \sur{Liu}}\email{liuhai21@mails.ucas.ac.cn}

\author[1]{\fnm{Tiande} \sur{Guo}}\email{tdguo@ucas.ac.cn}

\author*[1]{\fnm{Congying} \sur{Han}}\email{hancy@ucas.ac.cn}

\affil[1]{\orgdiv{ School of Mathematical Sciences}, \orgname{University of Chinese Academy of Sciences (UCAS)}, \orgaddress{\city{Beijing}, \postcode{100049}, \country{China}}}

%%==================================%%
%% Sample for unstructured abstract %%
%%==================================%%

\abstract{We propose a single-loop variance-reduced acceleration framework, which relates checkpoint update probabilities to momentum parameters, for solving the composite general convex problem where the smooth part has the finite-sum structure. Under the proposed framework, the growth rate of the momentum parameter is further altered, creating a novel continuous trade-off between acceleration and variance reduction, controlled by the key parameter $\alpha\in [0,1]$. A series of novel complexity is obtained, and some complexity of distinct known methods are rediscovered under the unified framework. When the mini-batch size is restricted due to the massive scale of the problem or the computational resource shortage, near-optimal complexity can still be achieved by choosing suitable $\alpha$ for any prefixed target accuracy. Analysis shows that although the considered gradient oracle is exact, acceleration comes with implicit price of heavier variance reduction, hence the obtained optimal $\alpha$ not necessarily corresponds to the largest allowable acceleration strength. Without prefixing the target accuracy, the proposed method achieves the near-optimal complexity $\Tilde{\mathcal{O}}(n+\sqrt{n}/\sqrt{\epsilon})$ to obtain an $\epsilon$-accurate solution under standard assumptions ($n$ is the number of components of the finite-sum), significantly improves upon previous best complexity $\mathcal{O}(n+n/\sqrt{\epsilon})$ of single-loop variance reduction methods, which does not exceed the complexity of the deterministic method FISTA. Numerical experiments demonstrate the efficiency of the proposed method and validate other theoretical findings.}

\keywords{finite-sum optimization, single-loop methods, variance reduction, acceleration}

%%\pacs[JEL Classification]{D8, H51}

\pacs[Mathematics Subject Classification]{90C06, 90C15, 90C25, 68Q25}

\maketitle

\section{Introduction}

In this work, we consider the following composite general convex optimization problem:
\begin{equation}
\label{problem setting}
    \min_{x\in \mathbb{R}^d} F(x):=f(x)+l(x)=\frac{1}{n}\sum_{i=1}^n f_i(x)+l(x),
\end{equation}
where $f(x):=\frac{1}{n}\sum_{i=1}^n f_i(x)$ with $f_i(x)$ being convex and smooth, and $l(x)$ is convex, lower semicontinuous (but possibly non-differentiable) function which admits an efficient proximal operator. Problems of this type are ubiquitous in statistics and machine learning such as regularized empirical risk minimization \cite{bottou2018optimization}.

When $n$ is large, the computation of $\nabla f(x)$ is expensive since it requires computing $n$ component gradients. Gradient-based methods (e.g., the forward-backward splitting algorithm \cite{lions1979splitting} and its inertial version \cite{nesterov2013gradient,beck2009fast}) require access to the full gradient at every iteration, hence are unpractical. By utilizing the finite-sum structure of (\ref{problem setting}), stochastic gradient methods use the gradient of one or a few  component(s) $\nabla f_i$ to construct the stochastic gradient estimator. For example, stochastic gradient descent (SGD) uses $\nabla f_i(x)$ ($i$ is chosen from $[n]:=\{1,2,\cdots,n\}$ according to the discrete uniform distribution) as an unbiased estimator of $\nabla f(x)$ in each iteration, hence enjoys a low per-iteration cost. However, due to the non-diminishing variance of its gradient estimator, the convergence speed of SGD is very slow. Even when the objective is smooth and strongly convex, SGD can only achieve sublinear convergence.

Variance reduction (VR) methods \cite{johnson2013accelerating,xiao2014proximal,nguyen2017sarah,allen2018katyusha,kovalev2020don,tan2016barzilai,liu2025improvement,gower2020variance} utilize history information to construct variance-vanishing stochastic gradient estimators, enjoying both low per-iteration cost and fast convergence speed. In fact, the convergence rate of VR methods improves significantly upon SGD, and matches that of traditional gradient-based methods. For example, in contrast to the slow convergence of SGD, VR methods achieve the linear convergence rate in the smooth strongly convex problems.

As one of the most noted and representative VR method, the vanilla SVRG method \cite{johnson2013accelerating} computes a full gradient $\nabla f(\Tilde{x})$ of a checkpoint $\Tilde{x}$, and uses it to construct the SVRG estimator
\begin{equation*}
    \Tilde{\nabla} f_{i_t}(x):=\nabla f_{i_t}(x)-\nabla f_{i_t}(\Tilde{x})+\nabla f(\Tilde{x}),
\end{equation*}
where $i_t$ is chosen from $[n]$ according to the discrete uniform distribution at $t$th iteration. Then the iterate is iterated with the formula $x_{t+1}=x_t-\eta\Tilde{\nabla}f_{i_t}(x_t)$, where $\eta>0$ is the step size. After $m$ iterations, the checkpoint is updated (e.g., it can be chosen as the average of last $m$ iterates) and the aforementioned procedure is repeated. Hence the vanilla SVRG method, along with many other VR methods \cite{konevcny2017semi,lei2020adaptivity,nguyen2017sarah,horvath2022adaptivity}, have nested algorithmic structure.

VR methods achieve well-known success in solving the composite general convex problem (\ref{problem setting}) under the incremental first-order oracle (IFO) framework \cite{agarwal2015lower}, which evaluates the efficiency of an algorithm by the number of component gradient $\nabla f_i(x)$ calculations needed to achieve an $\epsilon$-accurate solution. The vanilla SVRG relies on the strong convexity of $f$ to combine different outer iterations in the proof, hence cannot directly solve general convex problem theoretically. SVRG\textsuperscript{++} \cite{allen2016improved} resolves this issue by introducing the technique of doubling the inner loop length (i.e., $m_s=2^s\cdot m_0$ where $m_s$ is the inner loop length of the $s$th outer iteration), achieving the IFO complexity $\mathcal{O}(n\log(1/\epsilon)+1/\epsilon)$. 

Katyusha\textsuperscript{ns} \cite{allen2018katyusha} uses the Katyusha momentum \cite{allen2018katyusha} to mitigate possible negative effects of Nesterov acceleration in stochastic setting. It achieves the complexity $\mathcal{O}(n+n/\sqrt{\epsilon})$, and the Katyusha momentum turned out to be instrumental for the acceleration of VR methods in the composite general convex problem. MiG\textsuperscript{nsc}\cite{zhou2018simple} simplifies the iteration structure of Katyusha\textsuperscript{ns} and achieves essentially the same complexity $\mathcal{O}(n+n/\sqrt{\epsilon})$. Using the Katyusha momentum in inner iterations, ASVRG \cite{shang2018asvrg} achieves the complexity $\mathcal{O}(n\log(1/\epsilon))+\sqrt{n}/\sqrt{\epsilon})$, which nearly matches the theoretical lower bound $\Omega(n+\sqrt{n}/\sqrt{\epsilon})$ \cite{woodworth2016tight}. By utilization of the Katyusha momentum, the acceleration techniques in \cite{ghadimi2012optimal,ghadimi2013optimal}, and a multi-stage inner loop length selection scheme, Varag \cite{lan2019unified} also achieves the complexity $\mathcal{O}(n\log(n)+\sqrt{n}/\sqrt{\epsilon})$. VRADA \cite{song2020variance} incorporates the Katyusha momentum with a novel initialization technique, achieving the complexity $\mathcal{O}(n\log\log(n)+\sqrt{n}/\sqrt{\epsilon})$. Based on these methods, some researchers employ the AdaGrad idea \cite{duchi2011adaptive} to compute step sizes \cite{dubois2022svrg}, maintaining similar complexity under additional assumptions such as the compactness of the underlying set which the optimization variable belongs to.

As introduced above, some VR methods can achieve the near-optimal complexity $\widetilde{\mathcal{O}}(n+\sqrt{n}/\sqrt{\epsilon})$, where the $\widetilde{\mathcal{O}}$ notation hides logarithmic factors. However, all such methods have the nested algorithmic structure, which brings two major drawbacks. First, this complicated structure renders the implementation of the algorithm more troublesome. Second, these methods only have convergence guarantee for the checkpoint sequence. The checkpoint is usually not the last iterate of the last $m$ inner iterations\footnote{Theoretically, choosing the last iterate of last $m$ inner iterations to be the new checkpoint deteriorates the complexity. Hence existing methods usually choose the checkpoint with certain (weighted) average, although this is empirically undesirable.} and a new checkpoint can be obtained only after an entire outer iteration. Since the inner loop length $m$ can be very large, e.g., $m=\mathcal{O}(n)$, the limited convergence guarantee can be troublesome in large-scale problems.

In contrast, single-loop VR methods are easy to implement and typically provide convergence guarantee for the iterate sequence itself. However, the literature on single-loop VR methods for the composite general convex problem (\ref{problem setting}) is relatively limited. The complexity of single-loop VR methods are no better than $\mathcal{O}(n+n/\sqrt{\epsilon})$, which is exactly the complexity of FISTA \cite{beck2009fast}. Note that FISTA is an extension of Nesterov's accelerated gradient method \cite{nesterov1983method} for the composite optimization problem, and it requires full gradient computation in every iteration, hence is a gradient-based method.

In this line of works, L-SVRG \cite{kovalev2020don} and L-Katyusha \cite{kovalev2020don} are first single-loop VR methods. In L-SVRG, the checkpoint, denoted as $w_t$ (where $t$ is the iteration number), is updated in a randomized way. Concretely, at the end of the $t$th iteration, $w_{t+1}$ is replaced by the new iterate\footnote{Once the replacement occurs, a full gradient is computed at $w_{t+1}$.} with probability $p$ or remains $w_t$ otherwise, where $p\in [0,1]$ is a constant which takes a small value, e.g., $1/n$. However, L-SVRG and L-Katyusha require the smooth part $f$ to be strongly convex and the non-smooth part $l\equiv 0$, hence cannot solve the composite general convex problem (\ref{problem setting}). L2S \cite{li2020convergence} is a single-loop VR method that uses the SARAH gradient estimator \cite{nguyen2017sarah} and a constant checkpoint update probability, having the complexity $\mathcal{O}(n+\sqrt{n}/\epsilon)$. However, it requires that $l\equiv 0$ in (\ref{problem setting}) and the convergence guarantee is only for the expected squared gradient norm, not the expected functional gap $\mathbb{E}[F(\cdot)-\inf F]$ which is usually adopted.

Without using the Katyusha momentum, Driggs et al. \cite{driggs2022accelerating} proposed a universal single-loop acceleration framework for several frequently used variance reduced gradient estimators based on the Nesterov acceleration scheme. Under the composite general convex setting, their method has the convergence rate $\mathcal{O}(c_n/T^2)$, where $T$ is the number of total iterations and $c_n$ is an $n$-dependent constant. Curiously, using mini-batch cannot improve the complexity of this scheme, regardless of which estimator is employed. For example, when the SVRG estimator is used, as the mini-batch size $b$ grows, $c_n$ is decreased but the per-iteration cost is increased. As a result, the gain in the iteration complexity and the loss in the per-iteration cost cancels perfectly, leading to the complexity $\mathcal{O}(n+n/\sqrt{\epsilon})$ for any $b\in [n]$. The checkpoint update probability in \cite{driggs2022accelerating} is also a constant $p=b/n$. 

Note that setting $p$ to a constant is essential for the proof of \cite{driggs2022accelerating}, since Driggs et al. estimate the summation of the bias term and mean squared error of the gradient estimator respectively along the entire optimization process. The constant probability setting renders certain coefficient (in the proof) essentially the truncation of a geometric series hence upper-bounded and independent of the total iteration number $T$. If a non-constant probability setting is to be directly applied, a positive lower bound of the probabilities is likely to be required to prevent aforementioned coefficients from exploding. But in this way, the complexity arguably relies on the aforementioned lower bound in essential, hence cannot be improved in terms of the $n$-dependence.

After the completion of this manuscript (except for the experimental parts), we came to realize the article \cite{li2025sifar} in which the SIFAR algorithm is proposed. The SIFAR algorithm is a single-loop VR method with acceleration technique and the checkpoint update probability $p_t$ varies between different stages. \cite{li2025sifar} denotes the iteration in which the checkpoint is updated for the first time as $t_1$, which follows the geometric distribution. There are three stages of SIFAR: $p_t\equiv 1/(n+1)$ for the first stage $0\leq t\leq t_1$, $p_t=4/(t-t_1+3\sqrt{n})$ for the second stage $t_1<t\leq t_1+n+3-3\sqrt{n}$, and $p_t\equiv 4/(n+3)$ for the third stage $t>t_1+n+3-3\sqrt{n}$. 

The complexity $\mathcal{O}(n\min \{1+\log (1/(\epsilon\sqrt{n})),\log \sqrt{n}\}+\sqrt{n/\epsilon})$ is claimed in \cite{li2025sifar}. However, despite the appealing novelty and innovation of \cite{li2025sifar}, after careful consideration, we believe that the convergence analysis and complexity analysis concerning the general convex setting are incorrect (detailed discussions are deferred to Appendix \ref{comments on SIFAR}), and the corresponding convergence property and complexity are invalid.

Inspired by the aforementioned theoretical and practical merit of the single-loop methods, we take this line of research further in this work. Our main contributions are summarized as follows:
\begin{itemize}
    \item We propose a single-loop variance-reduced acceleration framework, where the checkpoint update probabilities are related to the momentum parameters. By further alternating the growth rate of the momentum parameters, we discover a novel continuous trade-off between acceleration and variance reduction, controlled by the key parameter $\alpha\in [0,1]$. A series of novel complexity are obtained, and some known complexity of methods using different optimization techniques are rediscovered under the proposed unified framework.
    \item When the mini-batch size is restricted, the proposed method still achieves near-optimal complexity with proper $\alpha$ choice for any prefixed target accuracy. Analysis reveals that even in the noise-absent case, acceleration comes with implicit price of heavier variance reduction, hence the proposed acceleration-variance reduction trade-off should be considered and can be beneficial in certain scenarios.
    \item Setting $\alpha=1$ and the mini-batch size $b=\mathcal{O}(\sqrt{n})$, we prove that the proposed method achieves the complexity $\Tilde{\mathcal{O}}(n+\sqrt{n}/\sqrt{\epsilon})$ under standard assumptions. The obtained complexity significantly improves upon previous best complexity $\mathcal{O}(n+n/\sqrt{\epsilon})$ of single-loop VR methods, and matches the lower bound $\Omega(n+\sqrt{n}/\sqrt{\epsilon})$ up to a logarithmic factor.
    \item Numerical experiments demonstrate the efficiency of the proposed method and validate other theoretical results.
\end{itemize}

The rest of this article is organized as follows: Section \ref{notations and assumptions} introduces notations, definitions and assumptions. Section \ref{the KH method} presents the proposed framework and some preliminary explanations. Section \ref{main results and related discussions} presents our theoretical findings and detailed discussions and interpretations concerning the theoretical results, and most of the proofs are deferred to the Appendix. Section \ref{experiments} gives the results of numerical experiments. Finally, Section \ref{conclusions} summarises the proposed idea and technique from a high-level perspective and relist the contributions briefly.

\section{Notations and Assumptions}
\label{notations and assumptions}

For a differentiable function $f$, its gradient at $x$ is denoted by $\nabla f(x)$. The Euclidean norm of a vector $x\in \mathbb{R}^{d}$ is denoted as $\Vert x\Vert$. The set of all non-negative integers is denoted as $\mathbb{N}$, and the set of all positive integers is denoted as $\mathbb{N}^+$. The set $\{1,2,\cdots,n \}$ is denoted by $[n]$. The symbol $\mathbb{E}$ denotes the expectation of a random element, and $\mathbb{E}_{X}$ denotes an expectation over the randomness of a random variable $X$ while conditioning on all other random variables. 

We define computational costs by making use of the IFO framework of \cite{agarwal2015lower}, where we assume that sampling an index $i$ and computing the pair $(f_i(x),\nabla f_i(x))$ incurs a unit of cost. In the composite general convex setting, $x$ is called an $\epsilon$-accurate solution iff $\mathbb{E} \left\{ F(x)-F(x^*)\right\} \leq \epsilon$, where $x^*$ is a solution to the problem (\ref{problem setting}). The IFO complexity is the minimum IFO calls needed to reach an $\epsilon$-accurate solution and is denoted by $C_{comp}$.

In this work, the following standard assumptions will be made throughout:

\noindent \textbf{A1} The optimization problem (\ref{problem setting}) has at least a solution, and we denote one fixed solution as $x^{*}$.

\noindent \textbf{A2} Functions $f_i: \mathbb{R}^d\rightarrow \mathbb{R}$ are $L$-smooth for some $L>0$:
\begin{equation*}
    f_i(y)\leq f_i(x)+\langle \nabla f_i(x),y-x\rangle+\frac{L}{2}\Vert y-x\Vert^2, \quad \forall x,y\in \mathbb{R}^d.
\end{equation*}

\noindent \textbf{A3} Functions $f_i: \mathbb{R}^d\rightarrow \mathbb{R}$ are convex:
\begin{equation*}
    f_i(y)\geq f_i(x)+\langle \nabla f_i(x),y-x\rangle, \quad \forall x,y\in \mathbb{R}^d.
\end{equation*}

\noindent \textbf{A4} The function $l$ is convex, lower semi-continuous (but possibly non-differentiable) which admits an efficient proximal operator.

\section{The Acceleration Framework}
\label{the KH method}

The proposed framework is formally described in Algorithm \ref{acceleration framework}. The three sequences formulation, consisting of ${x_t}, {y_t}, {z_t}$, is ubiquitous in accelerated gradient methods \cite{d2021acceleration} and can be interpreted by the linear coupling interpretation \cite{allen2017linear} as follows. Consider a simplified setting where $l\equiv 0$ so that $F\equiv f$ and the mini-batch size $b=n$ so that $\Tilde{g}_{t+1}=\nabla f(x_{t+1})$. Then the $\{z_t\}$ sequence is viewed and analysed as a mirror descent sequence, where an error term proportional to $\Vert \nabla f(x_{t+1})\Vert^2$ is introduced in its analysis. The $\{y_t\}$ iteration can be equivalently formulated as $y_{t+1}=y_t-\tau_t\alpha_t\eta \nabla f(x_{t+1})$, which further reduces to $y_{t+1}=y_t-\eta \nabla f(x_{t+1})$ when $\tau_t\alpha_t=1$. Hence $\{y_t\}$ is viewed and analysed as a gradient descent sequence, where an improvement proportional to $\Vert \nabla f(x_{t+1})\Vert^2$ is made in each iteration. Finally, $x_{t+1}$ linearly couples $z_t$ and $y_t$ to achieve acceleration, in the spirit of leveraging the aforementioned duality of the mirror descent and the gradient descent.

Algorithm \ref{acceleration framework} uses the mini-batch version of the SVRG estimator $\Tilde{g}_{t+1}$, where the full gradient $\nabla f(w_t)$ is required. The iterate $w_t$ is thus called the full gradient point, also called the checkpoint. In Algorithm \ref{acceleration framework}, $w_t$ is also involved in the coupling sequence through $x_{t+1}=\tau_t z_t+\xi w_t+(1-\xi-\tau_t)y_t$, and the $w_t$ in this formula is known as the Katyusha momentum. By coupling with $w_t$, $x_{t+1}$ can be attracted to $w_t$ so that the variance of the estimator $\Tilde{g}_{t+1}$ can be further reduced, enabling the acceleration technique to function well in the stochastic setting.

\begin{algorithm}[h]
\caption{The single-loop variance-reduced acceleration framework}
\begin{algorithmic}
\REQUIRE initial point $w_1=x_1=y_1=z_1$\\
\textbf{Parameters:} step size $\eta$, mini-batch size $b$, Katyusha momentum parameter $\xi$, momentum parameter sequences $\{\alpha_t \}_{t\in\mathbb{N}}$, $\{\tau_t \}_{t\in\mathbb{N}^{+}}$
\FOR{$t=1,2,\cdots$}
\STATE $x_{t+1}=\tau_t z_t+\xi w_t+(1-\xi-\tau_t)y_t$
\STATE Pick a subset $J_t\subset [n]$ according to the discrete uniform distribution on the set of all subsets of $[n]$ which has the cardinality $b$
\STATE $\Tilde{g}_{t+1}=\frac{1}{b}\left(\sum_{j\in J_t} \nabla f_j(x_{t+1})-\nabla f_j(w_t) \right)+\nabla f(w_t)$
\STATE $z_{t+1}=\arg\min_{z\in\mathbb{R}^d} \left\{\frac{1}{2\alpha_t\eta}\Vert z-z_t\Vert^2+\langle \Tilde{g}_{t+1},z\rangle+l(z) \right\}$
\STATE $y_{t+1}=x_{t+1}+\tau_t (z_{t+1}-z_t)$
\STATE $w^{t+1} = \begin{cases}
			y_t & \text{with probability } p_t\\
			w_t & \text{with probability } 1-p_t
			\end{cases}$
\STATE $p_t=\frac{\alpha_{t-1}^2-\alpha_t^2+\alpha_t+\xi\alpha_t^2}{\Tilde{\alpha}_0+\alpha_0^2-\alpha_{t}^2+\sum_{j=1}^{t}\alpha_j}$
\ENDFOR
\end{algorithmic}
\label{acceleration framework}
\end{algorithm}

The checkpoint $w_t$ is updated in a probabilistic way, which was first proposed by \cite{kovalev2020don}. Most single-loop VR methods, including the pioneering  work \cite{kovalev2020don}, use constant checkpoint update probability. To the best of our knowledge, the SVRLS method \cite{jiang2023adaptive} is the only method that uses varying update probability $p_t=1/(at+1)$ where $0< a<1$. However, SVRLS is a non-accelerated method that is analysed by an AdaGrad-type regret analysis under an extra compact feasible set condition, requires that the nonsmooth part $l\equiv 0$ in (\ref{problem setting}), and unfortunately fails to provide convergence guarantee for the iterate sequence itself\footnote{SVRLS thus does not have the last-iterate convergence guarantee.}.

The key feature of Algorithm \ref{acceleration framework} is the utilization of the proposed $p_t-\alpha_t$ correspondence, which is instrumental to the theoretical improvement of this work, and also leads to some novel insights into single-loop VR methods. These will be discussed concretely in the next section, where we present theoretical properties of Algorithm \ref{acceleration framework}.

\section{Main Results and Related Discussions}
\label{main results and related discussions}

The following lemma shows that Algorithm \ref{acceleration framework} is well-defined and is fundamental to the convergence analysis of Algorithm \ref{acceleration framework}. For example, it proves that in Algorithm \ref{acceleration framework} the parameter $p_t\in [0,1]$ hence is indeed a probability for all $t\in \mathbb{N}^{+}$, and that the coefficient of certain component of the Lyapunov function in the convergence analysis is non-negative for all $t\in \mathbb{N}^{+}$.

\begin{lemma}
\label{lemma, well-definedness}
For $t\in\mathbb{N}$, define
\begin{equation*}
\alpha_t = \begin{cases}
6, & 0\leq t\leq 16 \\
a_{\alpha} t^{\alpha}, & t\geq 17
\end{cases}
\quad \text{with} \quad
a_{\alpha} = \begin{cases}
6, & \alpha=0 \\
1+\frac{\sqrt{2}}{4}, & \alpha\in\left(0,1/2\right] \\
\frac{1}{3}, & \alpha\in \left(1/2,3/4\right] \\
\frac{1}{4}\left(\frac{17}{16}\right)^{\alpha-1}, & \alpha\in (3/4,1]
\end{cases}.
\end{equation*}
In Algorithm \ref{acceleration framework}, set $\xi=1/(bc)$ where $b\in[n]$ is the mini-batch size and \begin{equation*}
    c=\max \left\{2,b^{-1}\max \left\{6/5,(1-\alpha_{17}^{-1})^{-1} \right\} \right\}+1,
\end{equation*}
and $\Tilde{\alpha}_0\geq \xi\alpha_1^2$. Then in Algorithm \ref{acceleration framework}, $p_t\in [0,1]$ for $t\in \mathbb{N}^{+}$, $\Tilde{\alpha}_0+\alpha_0^2-\alpha_{t}^2+\sum_{j=1}^{t}\alpha_j\geq 0$ for $t\in\mathbb{N}$, and $\tau_t,\xi,(1-\xi-\tau_t)\in (0,1)$ for $t\in \mathbb{N}^{+}$. Moreover, there exists a constant $C>0$ which does not depend on any problem-dependent constants (e.g., $n$) such that\footnote{As is shown in the proof, $C$ can take a mild value, e.g., $C=5$.}
\begin{equation*}
    c\leq C, \ \text{for all} \ \alpha\in[0,1].
\end{equation*}
\end{lemma}
With Lemma \ref{lemma, well-definedness}, the convergence of Algorithm \ref{acceleration framework} can be guaranteed by the following theorem.
\begin{theorem}
\label{theorem, convergence}
With the $\{\alpha_t\}$ sequence set as in Lemma \ref{lemma, well-definedness}, set $\tau_t=1/\alpha_t$ for $t\in\mathbb{N}^{+}$, $\eta\leq 1/(c L+L)$ and fix any $b\in [n]$. For all $T\in\mathbb{N}^{+}$, after $T$ iterations, Algorithm \ref{acceleration framework} produces a point $w_T$ satisfying the following bound:
\begin{align}
    \alpha_T^2 \mathbb{E}\left[F(y_{T+1})-F(x^*) \right]+\left(\Tilde{\alpha}_0+\alpha_0^2-\alpha_{T}^2+\sum_{j=1}^{T}\alpha_j\right) \mathbb{E}\left[F(w_{T+1})-F(x^*) \right] \nonumber \\
    +\frac{1}{2\eta}\mathbb{E}\Vert z_{T+1}-x^*\Vert^2\leq\alpha_0^2 \left[F(y_1)-F(x^*) \right]+\Tilde{\alpha}_0\left[F(w_1)-F(x^*) \right]+\frac{1}{2\eta}\Vert z_1-x^*\Vert^2.
    \label{theorem 1, equation 0}
\end{align}
\end{theorem}
The setting of parameters of Algorithm \ref{acceleration framework} are determined in Lemma \ref{lemma, well-definedness} and Theorem \ref{theorem, convergence}, and it is fairly intuitive. First, $\xi=1/(b c)$ is the coefficient of the Katyusha momentum, which is constructed to cancel certain terms in the upper bound of the variance of the SVRG estimator, i.e., to further reduce the variance (in the sense that one can use a tighter upper bound on the variance of the SVRG estimator with the presence of the Katyusha momentum). Hence it is expected that $\xi$ becomes smaller when the variance of the SVRG estimator decreases, which is actually the case since $\xi$ is non-increasing (and likely to decrease) when the mini-batch size $b$ increases. Second, the largest allowable step size $\eta=1/(c L+L)$ is non-decreasing (and likely to increase) when $b$ increases by the definition of $c$, which is reasonable since larger step size is expected to work when the variance of the stochastic estimator decreases.

Note that the convergence guarantee is for the iterate sequence $\{w_t\}$ itself\footnote{By which we mean $F(w_t)$ converges to $F(x^*)$ in expectation.}, hence the last-iterate convergence is available. In the right hand side (RHS) of (\ref{theorem 1, equation 0}), the coefficient $\alpha_0^2$ equals $36$ and is an absolute constant, $\Tilde{\alpha}_0$ can also be set to an absolute constant, and $1/(2\eta)$ equals $L(c+1)/2$ when taking the largest allowable step size, where $c\leq C$ by Lemma \ref{lemma, well-definedness}. Hence the convergence rate is $\mathcal{O}(1/(\Tilde{\alpha}_0+\alpha_0^2-\alpha_{T}^2+\sum_{j=1}^{T}\alpha_j))$ and is $n$-independent\footnote{In Appendix \ref{proof of theorem 2}, we show that $\Tilde{\alpha}_0+\alpha_0^2-\alpha_T^2+\sum_{j=1}^T \alpha_j$ is lower bounded by a constant multiple of $T^{\alpha+1}$.}, which is in contrast to Theorem $5$ in \cite{driggs2022accelerating}, where the convergence rate has an $n$-dependent constant and the order of $n$ can be very large\footnote{In Theorem $5$ of \cite{driggs2022accelerating}, using the same notation as \cite{driggs2022accelerating}, the dominant constant of the convergence rate is either $\mathcal{O}(\nu^2)$ or $\mathcal{O}(c)$, which are both $n$-dependent.}.

With Theorem \ref{theorem, convergence}, one can obtain the number of iterations $T(\epsilon)$ needed to reach an $\epsilon$-accurate solution. We further estimate $\sum_{t=1}^{T(\epsilon)} \left(np_t+b \right)$ to obtain the IFO complexity of Algorithm \ref{acceleration framework}.

\begin{theorem}
\label{theorem, complexity}
Under the setting of Lemma \ref{lemma, well-definedness} and Theorem \ref{theorem, convergence}, for any fixed $\alpha\in [0,1]$ in the definition of $\{\alpha_t\}_{t\in\mathbb{N}}$ in Lemma \ref{lemma, well-definedness}, the convergence rate of Algorithm \ref{acceleration framework} is $\mathcal{O}(1/T^{\alpha+1})$, the per-iteration IFO call in expectation is $\mathcal{O}(b+n/(bt^{1-\alpha})+n/t)$ for all $t\geq 17$. Without loss of generality, assume that the target accuracy $\epsilon\in (0,1)$, define $\hat{\alpha}=\max\{\log(2)/\log(\lceil 1/\epsilon\rceil),\log(2)/\log(n) \}$, then the IFO call in expectation needed by Algorithm \ref{acceleration framework} to reach an $\epsilon$-accurate solution is
\begin{equation}
\label{theorem 2, equation 0}
    \mathbb{E}C_{\text{comp}}=\begin{cases}
    \mathcal{O}\left(\frac{b}{\epsilon^{1/(\alpha+1)}}+n\min\left\{\log(n),\log(1/\epsilon)\right\} \right),  &\alpha\in\left(0,\min \left\{\hat{\alpha},1/10\right\}\right], \\
    \mathcal{O}\left(\frac{b}{\epsilon^{1/(\alpha+1)}}+\left(\frac{n}{b}\frac{1}{\epsilon^{\alpha/(\alpha+1)}}+n\right)\min\left\{\log(n),\log(1/\epsilon)\right\} \right),  &\alpha\in \left(\min\left\{\hat{\alpha},1/10\right\},1\right].
    \end{cases}
\end{equation}
In the statements above, $T\in\mathbb{N}^{+}$ is the total iteration number, $t\in [T]$ is the iteration counter, $n\in\mathbb{N}^{+}$ is the number of smooth component functions in (\ref{problem setting}), $b\in [n]$ is the mini-batch size.
\end{theorem}
Note that the sequence $\{\mathbb{E}[F(y_t)]\}_{t}$ also converges to $F(w^*)$ with convergence rate $\mathcal{O}(1/t^{2\alpha})$, and when $\alpha=1$, it has the same convergence rate with the sequence $\{\mathbb{E}[F(w_t)]\}_{t}$.

In the finite-sum optimization, one of the most fundamental topics is the trade-off between the convergence rate and the per-iteration IFO cost, since essentially it is these two factors that determines the theoretical complexity of an optimization method. Handling this trade-off typically involves using different gradient estimators, using different optimization techniques, and even changing the algorithmic structure. For example, using full gradient gives the proximal gradient descent (PGD) a fast iteration convergence rate and high per-iteration IFO cost. While using the SGD estimator gives the proximal stochastic gradient descent (PSGD) contradictory characteristics. On the algorithmic structure side, note that the majority of variance reduction methods are nested, while PGD and PSGD are single-loop methods.

However, based on the proposed $p_t-\alpha_t$ correspondence, we can achieve this trade-off under the unified framework of Algorithm \ref{acceleration framework}, by changing the growth rate of the $\{\alpha_t \}_{t\in\mathbb{N}}$ sequence. Concretely, when the acceleration strength (determined by $\alpha$) increases, the convergence rate $\mathcal{O}(1/T^{\alpha+1})$ becomes faster, and the per-iteration cost $\mathcal{O}(b+n/(b t^{1-\alpha})+n/t)$ increases at the same time.

As mentioned above, less aggressive momentum acceleration (i.e., smaller $\alpha$) accompanies with smaller checkpoint update probability (i.e., smaller $p_t$), hence slower convergence rate can be compensated with lower per-iteration IFO cost in Algorithm \ref{acceleration framework}, showing a novel continuous trade-off between the convergence rate and the per-iteration cost. The key parameter $\alpha\in [0,1]$ is thus a weight that controls the proposed acceleration-variance reduction trade-off, through which $\alpha$ essentially controls the trade-off between the convergence rate and the per-iteration cost. Two favourable consequences of this mechanism are listed as follows.

First, Algorithm \ref{acceleration framework} recovers some existing IFO complexity of very different algorithms under a unified framework. Consider the $b=1$ case. When $\alpha=0$ (i.e., $\alpha_t\equiv 6$ for all $t\in\mathbb{N}$), the complexity of Algorithm \ref{acceleration framework} is $\mathcal{O}(n\log(1/\epsilon)+1/\epsilon)$, recovering the complexity of SVRG\textsuperscript{++} while removing some elusive requirements concerning the knowledge of the solution $x^*$ needed by SVRG\textsuperscript{++}. When $\alpha=1$ (i.e., $\alpha_t=\mathcal{O}(t)$), the complexity of Algorithm \ref{acceleration framework} is $\mathcal{O}(n+n/\sqrt{\epsilon})$, recovering the complexity of FISTA. Note that SVRG\textsuperscript{++} is a nested algorithm, uses the SVRG estimator, and employs the doubling technique, while FISTA is a single-loop algorithm, uses the full negative gradient as the iteration direction, and employs the Nesterov acceleration technique.

Second, we discover a series of novel complexity due to the fact that $\alpha\in [0,1]$ and can vary continuously. Consider the scenario where $n$ in (\ref{problem setting}) is very large or the computational resources are shorted, then the mini-batch size $b$ may be restricted to a small number. We show that when the target accuracy $\epsilon$ is prefixed and known, Algorithm \ref{acceleration framework} can achieve near-optimal complexity by choosing suitable $\alpha$ according to the knowledge of $n$ and $\epsilon$, under the restricted $b$ setting and regardless of the magnitude of $\epsilon$.

Consider the $b=1$ case for Algorithm \ref{acceleration framework}. When $n=10^4, \epsilon=10^{-12}$, the complexity of SVRG\textsuperscript{++}, FISTA, Algorithm \ref{acceleration framework} ($\alpha=1/2$), and the lower bound $\Omega(n+\sqrt{n}/\sqrt{\epsilon})$ are $\mathcal{O}(10^{12})$, $\mathcal{O}(10^{10})$, $\mathcal{O}(10^8)$ and $\mathcal{O}(10^8)$, respectively. In fact, given $n$ and $\epsilon\in (0,1)$ satisfying
\begin{equation}
\label{epsilon needed, condition}
    n< 1/\epsilon,
\end{equation}
for prefixed constants $C_1\in \left[1,+\infty\right)$ and $C_2\in \left[1,+\infty\right)$ (e.g., $C_1=C_2=2$), define\footnote{We require that the choice of $C_2$ makes $\Delta_2$ strictly positive. Under the condition (\ref{epsilon needed, condition}), this requirement can always be satisfied by choosing $C_2$ sufficiently close to $1$.}
\begin{equation*}
     \Delta=[\Delta_1,\Delta_2]=\left[\frac{1-\frac{\log(n)}{\log(1/\epsilon)}-2\frac{\log(C_1)}{\log(1/\epsilon)}}{1+\frac{\log(n)}{\log(1/\epsilon)}+2\frac{\log(C_1)}{\log(1/\epsilon)}}, \frac{1-\frac{\log(n)}{\log(1/\epsilon)}+2\frac{\log(C_2)}{\log(1/\epsilon)}}{1+\frac{\log(n)}{\log(1/\epsilon)}-2\frac{\log(C_2)}{\log(1/\epsilon)}}\right],
\end{equation*}
one can always choose $\alpha\in [0,1]$ satisfying
\begin{equation}
\label{epsilon needed, alpha choice}
    \alpha\in \begin{cases}
    \left[\max\left\{0,\Delta_1\right\},\Delta_2 \right]&\quad  \text{if} \ \Delta_2\leq \min\left\{\hat{\alpha},1/10 \right\}, \\
    \left[\max\left\{\min\left\{\hat{\alpha},1/10 \right\},\Delta_1\right\},\min\left\{1,\Delta_2\right\}\right] &\quad \text{if} \ \Delta_2>\min\left\{\hat{\alpha},1/10 \right\},
    \end{cases}
\end{equation}
then direct calculation gives
\begin{equation*}
    \frac{1}{\epsilon^{1/(\alpha+1)}}\leq C_1\frac{\sqrt{n}}{\sqrt{\epsilon}}, \ \frac{n}{\epsilon^{\alpha/(\alpha+1)}}\leq C_2\frac{\sqrt{n}}{\sqrt{\epsilon}}.
\end{equation*}
Hence the complexity in expectation of Algorithm \ref{acceleration framework} matches the theoretical lower bound up to a logarithmic factor under the $\alpha$ setting (\ref{epsilon needed, alpha choice}). This observation can be extended and summarised in the following corollary.

\begin{corollary}
\label{corollary, epsilon needed}
Under the setting and assumptions of Theorem \ref{theorem, complexity}, when the condition (\ref{epsilon needed, condition}) holds, by choosing $b=1$ and $\alpha$ satisfying (\ref{epsilon needed, alpha choice}), the complexity in expectation of Algorithm \ref{acceleration framework} is $\widetilde{\mathcal{O}}(n+\sqrt{n}/\sqrt{\epsilon})$. When the condition (\ref{epsilon needed, condition}) does not hold, by choosing $b=1$ and $\alpha=0$, the complexity of Algorithm \ref{acceleration framework} is $\Tilde{\mathcal{O}}(n)$. Hence, Algorithm \ref{acceleration framework} achieves near-optimal complexity in both cases under the restricted $b=1$ setting.
\end{corollary}
\begin{proof}
Since $\Delta_1<\Delta_2$ and $\Delta_2>0$ under the condition (\ref{epsilon needed, condition}) and the appropriate $C_2$ choice described in the footnote, the interval $\Delta$ is well-defined, and the intervals in (\ref{epsilon needed, alpha choice}) are well-defined and contained in $[0,1]$. Direct calculation shows that $\alpha\geq \Delta_1$ implies $1/\epsilon^{1/(\alpha+1)}\leq C_1\sqrt{n}/\sqrt{\epsilon}$ and $\alpha\leq \Delta_2$ implies $n/\epsilon^{\alpha/(\alpha+1)}\leq C_2\sqrt{n}/\sqrt{\epsilon}$. If $\Delta_2\leq \min\left\{\hat{\alpha},1/10\right\}$, under the first choice in (\ref{epsilon needed, alpha choice}), we have $\mathbb{E}C_{\text{comp}}$ takes the first form in (\ref{theorem 2, equation 0}) and $1/\epsilon^{1/(\alpha+1)}\leq \sqrt{n}/\sqrt{\epsilon}$, hence $\mathbb{E}C_{\text{comp}}=\widetilde{\mathcal{O}}(n+\sqrt{n}/\sqrt{\epsilon})$. Similarly, one sees that if $\Delta_2>\min\left\{\hat{\alpha},1/10 \right\}$, we have $\mathbb{E}C_{\text{comp}}=\widetilde{\mathcal{O}}(n+\sqrt{n}/\sqrt{\epsilon})$ under the second $\alpha$ choice in (\ref{epsilon needed, alpha choice}). In all, under the setting of this corollary, $\alpha$ satisfying (\ref{epsilon needed, alpha choice}) always exists and such choice of $\alpha$ yields the complexity $\widetilde{\mathcal{O}}(n+\sqrt{n}/\sqrt{\epsilon})$.

When $b=1$ and $\alpha=0$, by Theorem \ref{theorem, complexity}, the complexity of Algorithm \ref{acceleration framework} is $\mathcal{O}(1/\epsilon+n\log(1/\epsilon))$, which is also $\mathcal{O}(n+n\log(n))$ when $n\geq 1/\epsilon$. Under the condition $n\geq 1/\epsilon$, the first term in the lower bound $\Omega(n+\sqrt{n}/\sqrt{\epsilon})$ is dominant, hence the complexity of Algorithm \ref{acceleration framework} matches the lower bound up to a logarithmic factor in this case.
\end{proof}
For ease of discussion, we take $C_2=1$ in Corollary \ref{corollary, epsilon needed}, assume that $n$ in (\ref{problem setting}) is given and focus on the $n<1/\epsilon$ case. Some comments on Corollary \ref{corollary, epsilon needed} are listed as follows. First, we give intuitive explanations for the choice (\ref{epsilon needed, alpha choice}). The largest possible interval in (\ref{epsilon needed, alpha choice}) under the $\Delta_2\leq \min\left\{\hat{\alpha},1/10 \right\}$ case is $\left[0,\min\left\{\hat{\alpha},1/10\right\}\right]$, and the largest possible interval in (\ref{epsilon needed, alpha choice}) under the other case is $\left[\min\left\{\hat{\alpha},1/10\right\},1\right]$, which is separated from the former interval by $\min\left\{\hat{\alpha},1/10\right\}$. Since $n$ is given and $C_2=1$, $\Delta_2$ is smaller when $\log(n)/\log(1/\epsilon)$ is larger, i.e., $\epsilon$ is larger. Hence the selection rule (\ref{epsilon needed, alpha choice}) essentially suggests using mild acceleration and small checkpoint update probability when the target accuracy is low (i.e., $\epsilon$ is relatively large), and using aggressive acceleration and large checkpoint update probability when the target accuracy is high. 

To understand the strategy, one may begin by investigating Theorem \ref{theorem, complexity}. The convergence rate of $\mathbb{E}[F(w_{T})-F(x^*)]$ is $\mathcal{O}(1/T^{\alpha+1})$, hence it requires $\mathcal{O}(1/\epsilon^{1/(\alpha+1)})$ iterations to reach an $\epsilon$-accurate solution, which constitutes the first part of the complexity in (\ref{theorem 2, equation 0}) when neglecting the logarithmic factor. The second part of the complexity is $n/\epsilon^{\alpha/(\alpha+1)}$, which comes from computing full gradient at checkpoints. Note that as $\alpha$ increases, although the convergence rate improves and $1/\epsilon^{\alpha+1}$ decreases, but $n/\epsilon^{\alpha/(\alpha+1)}$ increases due to the increase of the checkpoint update probability $p_t$. Since $n/\epsilon^{\alpha/(\alpha+1)}$ contains a large factor $n$, it is desirable to have lower per-iteration cost in expectation (related to $p_t$) while maintaining the convergence rate not too slow.

The ideal trade-off between the convergence rate and the per-iteration cost depends on the relative scale between $n$ and $1/\epsilon$. In the low accuracy regime where $\epsilon$ is large, the algorithm should have relatively low per-iteration cost and can tolerate relatively slow convergence rate, and vise versa in the high accuracy regime where $\epsilon$ is small.

Second, Corollary \ref{corollary, epsilon needed} proves that more aggressive acceleration (i.e., larger $\alpha_t$) does not necessarily lead to better complexity in Algorithm \ref{acceleration framework} under the mini-batch size constraint $b=1$, while it is the trade-off between acceleration and variance reduction that actually matters. Since we take $C_2=1$, then $\Delta_2<1$ under the condition (\ref{epsilon needed, condition}), which means that the largest allowable acceleration $\alpha_t=\mathcal{O}(t)$ (which corresponds to $\alpha=1$) can not provide best achievable complexity in either cases of (\ref{epsilon needed, alpha choice})\footnote{Note that $C_2=1$ is a desirable mild constant factor in the complexity. One can also easily check that $\Delta_2<1$ for a wide range of $\epsilon$ when $C_2>1$ takes a mild constant value.}.

This phenomenon somehow contradicts with the conventional choice in acceleration methods for solving noise-absent problems. Typically, since the convergence rate of an acceleration method is $1/\alpha_T^2$ where $T$ is the total iteration number, larger $\alpha_t$ setting is favourable as long as it is allowable. 

In stochastic optimization problems where the gradient oracle contains stochastic noise, large momentum parameters may fail since the complexity explodes or the convergence fails. It is in such cases that the momentum parameter is then reduced in order to adapt to the inexactness of the gradient estimator \cite{gupta2024nesterov}. For another example, when the gradient estimate satisfies the absolute error condition, the convergence bound of accelerated methods consists of the vanishing term and the accumulated error \cite{liu2025nonasymptoticanalysisacceleratedmethods}. Since more aggressive acceleration decreases the vanishing term while increases the accumulated error, the acceleration strength is tuned to achieve the ideal trade-off in the finite-horizon setting where the iteration number is prefixed.

In contrast, in the considered noise-absent finite-sum problem, all $\alpha\in [0,1]$ settings are allowable in Algorithm \ref{acceleration framework} and there is no accumulated error term, but we actively choose to reduce the acceleration strength and successfully get compensated by the reduction of the per-iteration cost in expectation. The analysis shows that in Algorithm \ref{acceleration framework}, acceleration comes with implicit price of heavier variance reduction (achieved by more per-iteration cost), hence the trade-off should be considered in certain scenarios even when the gradient oracle is noise-free.

Finally, Algorithm \ref{acceleration framework} achieves the near-optimal complexity in expectation under standard assumptions, which significantly improves upon the previous best $\mathcal{O}(n+n/\sqrt{\epsilon})$ complexity of single-loop VR methods \cite{driggs2022accelerating} and matches the lower bound $\Omega(n+\sqrt{n}/\sqrt{\epsilon})$ up to a logarithmic factor.

\begin{corollary}
\label{corollary, epsilon not needed}
Under the setting of Theorem \ref{theorem, complexity}, setting $\alpha=1$ and $b=\mathcal{O}(\sqrt{n})$, then
\begin{equation*}
    \lim_{t\rightarrow \infty}\mathbb{E}\left[F(y_t)-F(x^*)\right]=0, \ \lim_{t\rightarrow \infty}\mathbb{E}\left[F(w_t)-F(x^*)\right]=0,
\end{equation*}
the complexity in expectation of Algorithm \ref{acceleration framework} (for both sequences $\{\mathbb{E}[F(y_t)]\}_{t}$ and $\{\mathbb{E}[F(w_t)]\}$) is $\widetilde{\mathcal{O}}(n+\sqrt{n}/\sqrt{\epsilon})$.
\end{corollary}

\section{Numerical Experiments}
\label{experiments}

The results of some numerical experiments are given to demonstrate the efficiency of the proposed method and validate other theoretical findings. We consider the following $l_1$-regularized logistic regression problem
\begin{equation}
\label{regularized logistic regression problem}
    \mathop{\rm{min}}_{x\in \mathbb{R}^d} \quad F(x)=\frac{1}{n}\sum_{i=1}^n \log [1+\exp(-b_ia_i^Tx)]+\frac{\lambda}{2}\Vert x\Vert_{1},
\end{equation}
where $a_i\in \mathbb{R}^d$ and $b_i\in \{-1,1 \}$ are the feature vector and class label of the $i$th sample, respectively, and $\lambda>0$ is a weighting parameter. The adopted binary classification datasets \textit{ijcnn1, w8a, rcv1} are downloaded from the  LIBSVM\footnote{Online available at \url{https://www.csie.ntu.edu.tw/\string~ cjlin/libsvmtools/datasets/.}} database. All the tested algorithms are run $10$ times, and the best performance of these algorithms is reported.

\subsection{Comparison with Other Methods}

In this subsection, we compare Algorithm \ref{acceleration framework} with FISTA, Katyusha (for non-strongly convex problems) and VRADA for solving (\ref{regularized logistic regression problem}), where $\lambda=10^{-4}$. The Euclidean norm of each data vector in the datasets are normalized to be $1$, and hence $0.25$ is an upper bound of the smoothness parameter $L$ of the smooth part of (\ref{regularized logistic regression problem}).

For the four algorithms we compare, the common parameter to tune is the parameter with respect to (w.r.t.) the Lipschitz constant, which is tuned in $\{0.0125,0.025,0.05,0.1,0.25,0.5 \}$. The inner loop length $m$ of Katyusha and VRADA are set as $m=2n$, as suggested by \cite{allen2018katyusha, song2020variance}. For Algorithm \ref{acceleration framework}, we set $\alpha=1$ and the mini-batch size $b$ as $b=\lceil 1/\sqrt{n}\rceil$, as required in Corollary \ref{corollary, epsilon not needed}. Other parameters of Algorithm \ref{acceleration framework} are set strictly according to the theoretical requirements, where $\Tilde{\alpha}_0$ is set as $\Tilde{\alpha}_0=\max\{\xi\alpha_1^2,36 \}$ which satisfies the theoretical requirement in Lemma \ref{lemma, well-definedness} and prevents frequent computation of full gradient at the beginning of the optimization process.

The comparison results of tested algorithms are shown in Fig. \ref{fig:between-comparison}. In Fig. \ref{fig:between-comparison}, the $x$-axis denotes the number of full gradient evaluations, the $y$-axis denotes the objective function value at the checkpoints of the VR methods (and the iterates of the FISTA method).

It can be seen from Fig. \ref{fig:between-comparison} that, Algorithm \ref{acceleration framework} converges to a moderate accuracy faster than other tested algorithms, saving a substantial amount of full gradient evaluations compared to FISTA, Katyusha, and VRADA, in the sense that the ratio between required epochs of Algorithm \ref{acceleration framework} and that of other methods is small. This feature is favourable in certain scenarios, e.g., in large-scale machine learning tasks where a solution of moderate accuracy is required.

\begin{figure}

\begin{subfigure}[h]{0.32\textwidth}
\includegraphics[width=0.95\linewidth]{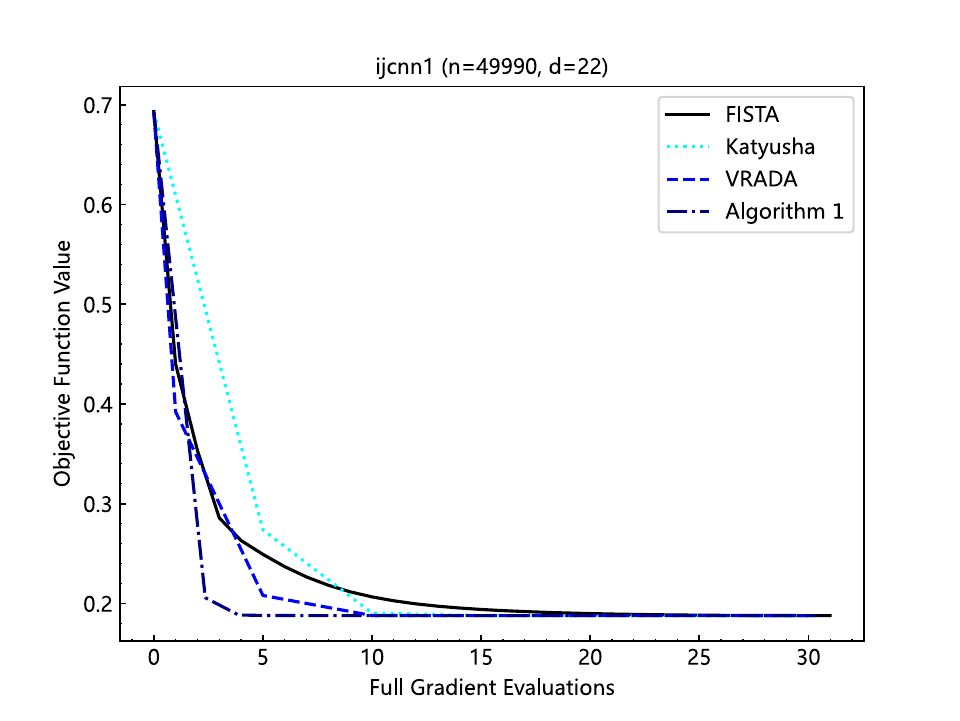}
\caption{ijcnn1}
\label{fig:comparison_ijcnn1}
\end{subfigure}
\begin{subfigure}[h]{0.32\textwidth}
\includegraphics[width=0.95\linewidth]{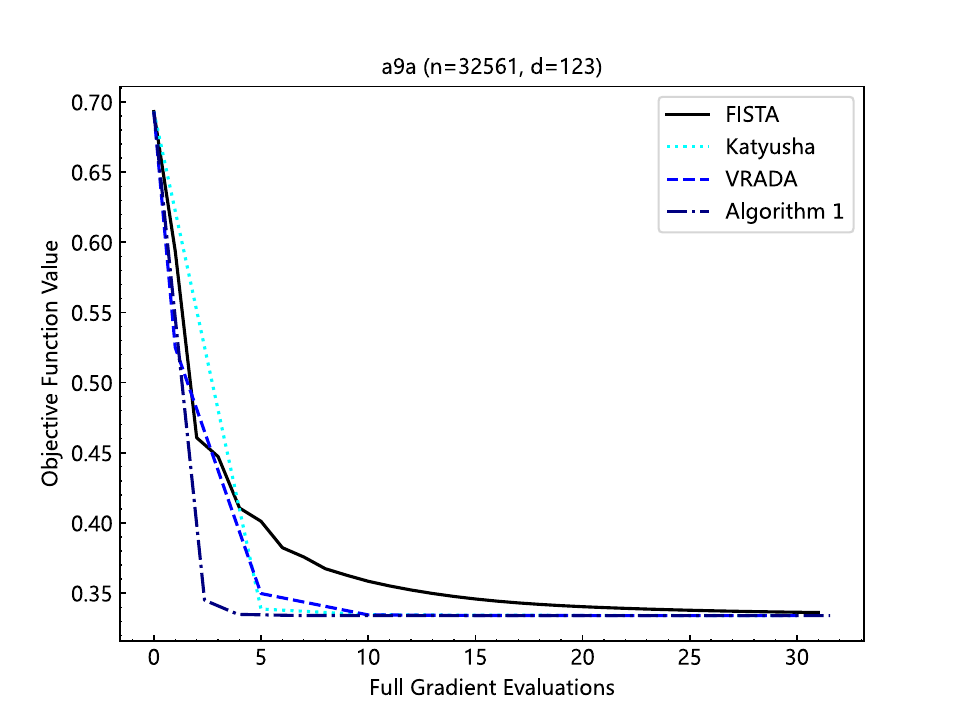}
\caption{a9a}
\label{fig:comparison_w8a}
\end{subfigure}
\begin{subfigure}[h]{0.32\textwidth}
\includegraphics[width=0.95\linewidth]{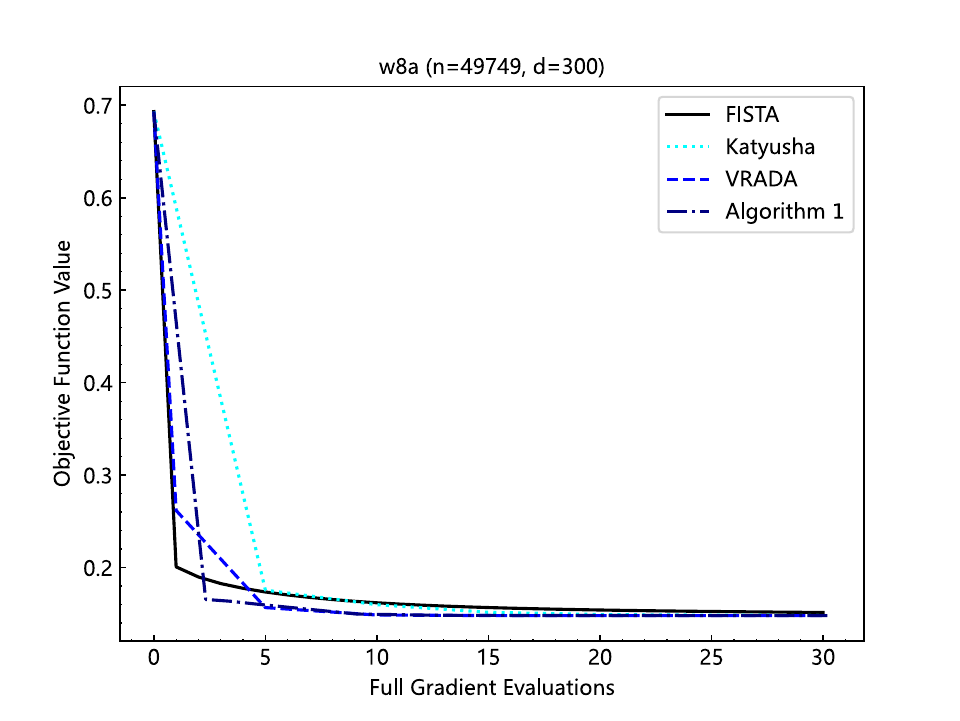}
\caption{w8a}
\label{fig:comparison_rcv1}
\end{subfigure}

\caption{Performance comparison for solving regularized logistic regression among different algorithms}
\label{fig:between-comparison}
\end{figure}

\subsection{Empirical Investigation of the Acceleration-Variance Reduction Trade-Off}

In this subsection, we investigate the proposed acceleration-variance reduction trade-off numerically. Algorithm \ref{acceleration framework} with different $\alpha$ settings are used for solving (\ref{regularized logistic regression problem}), where $\lambda=10^{-4}$. The Euclidean norm of each data vector in the datasets are normalized to be $1$, and hence $0.25$ is an upper bound of the smoothness parameter $L$ of the smooth part of (\ref{regularized logistic regression problem}).

For Algorithm \ref{acceleration framework} with different $\alpha$ settings, we use the Lipschitz constant estimate $0.25$ and set the mini-batch size $b$ as $b=1$. Other parameters are set the same and strictly satisfy the theoretical requirements of Corollary \ref{corollary, epsilon needed} to highlight and investigate the effect of the key parameter $\alpha$.

The comparison results are shown in Fig. \ref{fig:self-comparison}. In Fig. \ref{fig:self-comparison}, the $x$-axis denotes the number of full gradient evaluations, the $y$-axis denotes the gap between the objective function value at the checkpoint and the approximate minimum of the objective function value. The intersections on the $x$-axis indicate required full gradient evaluations to reach an $\epsilon$-accurate solution, where $\epsilon$ is the prefixed target accuracy.

In Fig. (\ref{fig:self_comparison_low_accuracy_ijcnn1}), (\ref{fig:self_comparison_low_accuracy_a9a}), and (\ref{fig:self_comparison_low_accuracy_w8a}), the target accuracy is $\epsilon=1/\sqrt{n}$ where $n$ is the number of samples of the dataset. In this low-accuracy case, the optimal $\alpha$ value suggested in Corollary \ref{corollary, epsilon needed} is $\alpha=0$. In Fig. (\ref{fig:self_comparison_high_accuracy_ijcnn1}), (\ref{fig:self_comparison_high_accuracy_a9a}), and (\ref{fig:self_comparison_high_accuracy_w8a}), the target accuracy is $\epsilon=1/(5n)$. In this case, the $\alpha$ is set as
\begin{equation*}
    \alpha=\alpha^{*}=\frac{1-\frac{\log(n)}{\log(1/\epsilon)}}{1+\frac{\log(n)}{\log(1/\epsilon)}},
\end{equation*}
which satisfies the theoretical requirement of Corollary \ref{corollary, epsilon needed}.

It can be seen from Fig. \ref{fig:self-comparison} that, Algorithm \ref{acceleration framework} with theoretically suggested (small) $\alpha$ setting significantly outperforms Algorithm \ref{acceleration framework} with large $\alpha$ setting. Algorithm \ref{acceleration framework} with large $\alpha$ setting, i.e., aggressive acceleration strength, requires relatively frequent full gradient computation. This heavier variance reduction, although allows more aggressive acceleration, ultimately leads to inferior empirical performance for the given accuracy, which validates our theoretical results in Corollary \ref{corollary, epsilon needed}. It also follows that it is beneficial to consider the proposed acceleration-variance reduction trade-off rather than acceleration alone in certain scenarios.

\begin{figure}

\begin{subfigure}[h]{0.32\textwidth}
\includegraphics[width=0.95\linewidth]{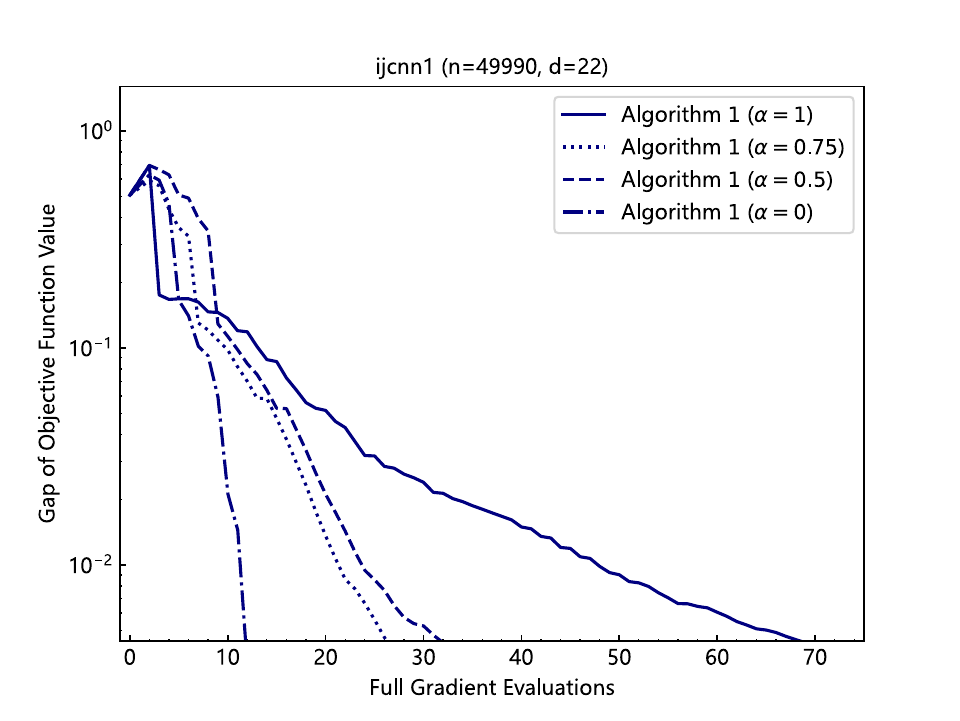}
\caption{ijcnn1}
\label{fig:self_comparison_low_accuracy_ijcnn1}
\end{subfigure}
\begin{subfigure}[h]{0.32\textwidth}
\includegraphics[width=0.95\linewidth]{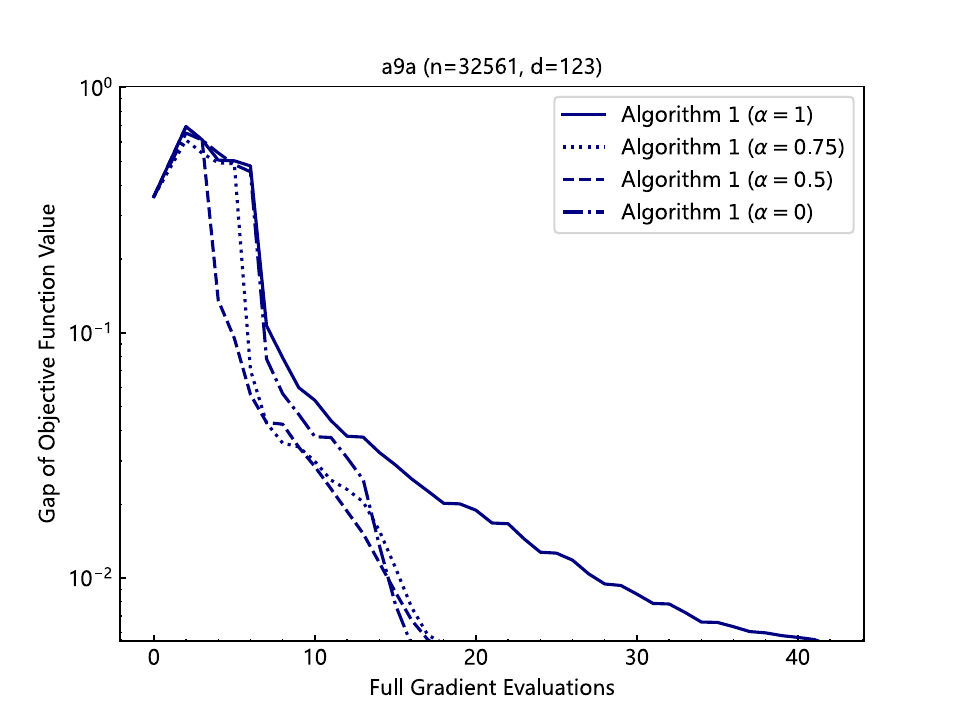}
\caption{a9a}
\label{fig:self_comparison_low_accuracy_a9a}
\end{subfigure}
\begin{subfigure}[h]{0.32\textwidth}
\includegraphics[width=0.95\linewidth]{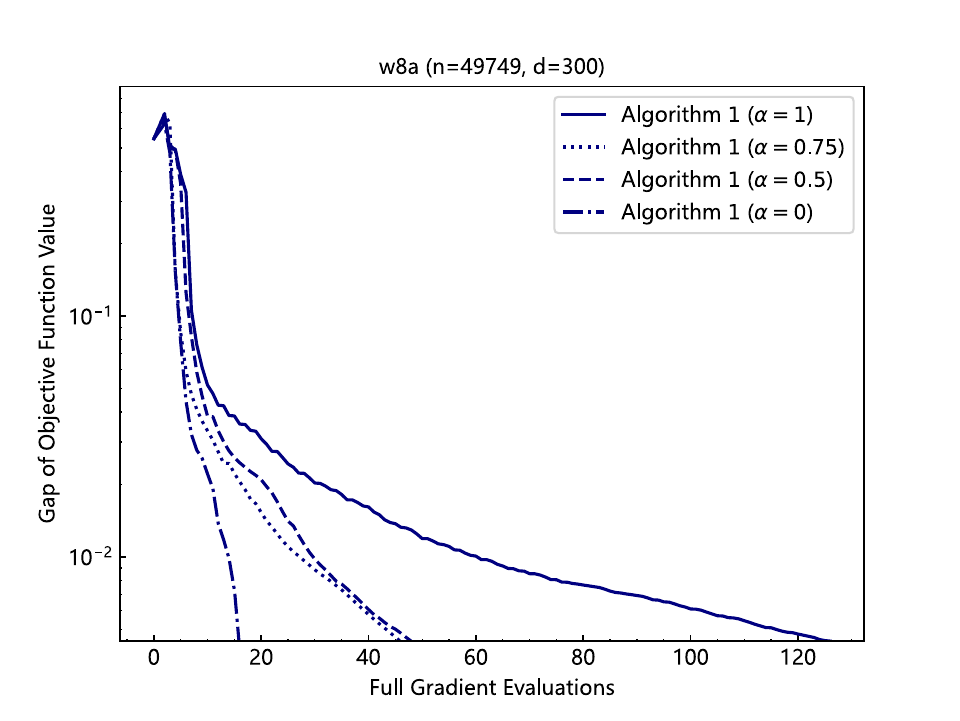}
\caption{w8a}
\label{fig:self_comparison_low_accuracy_w8a}
\end{subfigure}
\begin{subfigure}[h]{0.32\textwidth}
\includegraphics[width=0.95\linewidth]{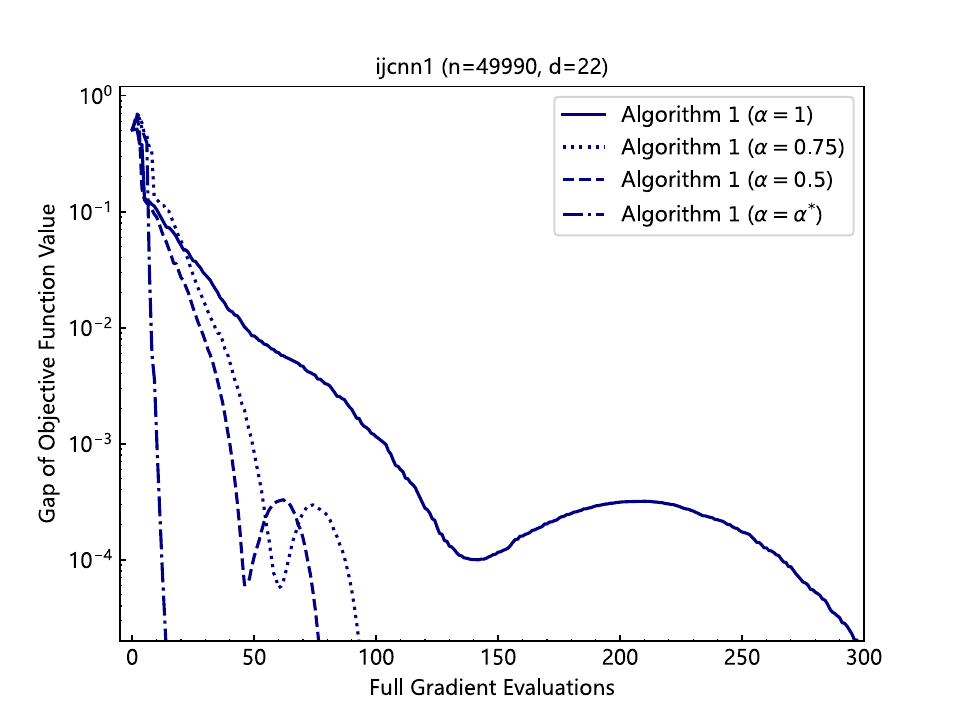}
\caption{ijcnn1}
\label{fig:self_comparison_high_accuracy_ijcnn1}
\end{subfigure}
\begin{subfigure}[h]{0.32\textwidth}
\includegraphics[width=0.95\linewidth]{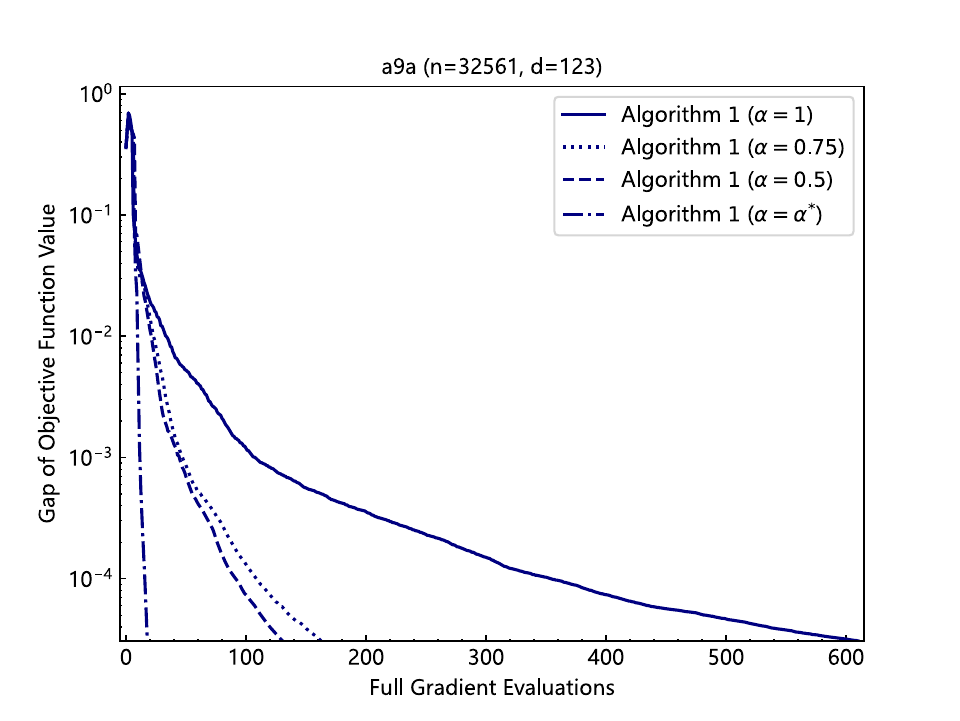}
\caption{a9a}
\label{fig:self_comparison_high_accuracy_a9a}
\end{subfigure}
\begin{subfigure}[h]{0.32\textwidth}
\includegraphics[width=0.95\linewidth]{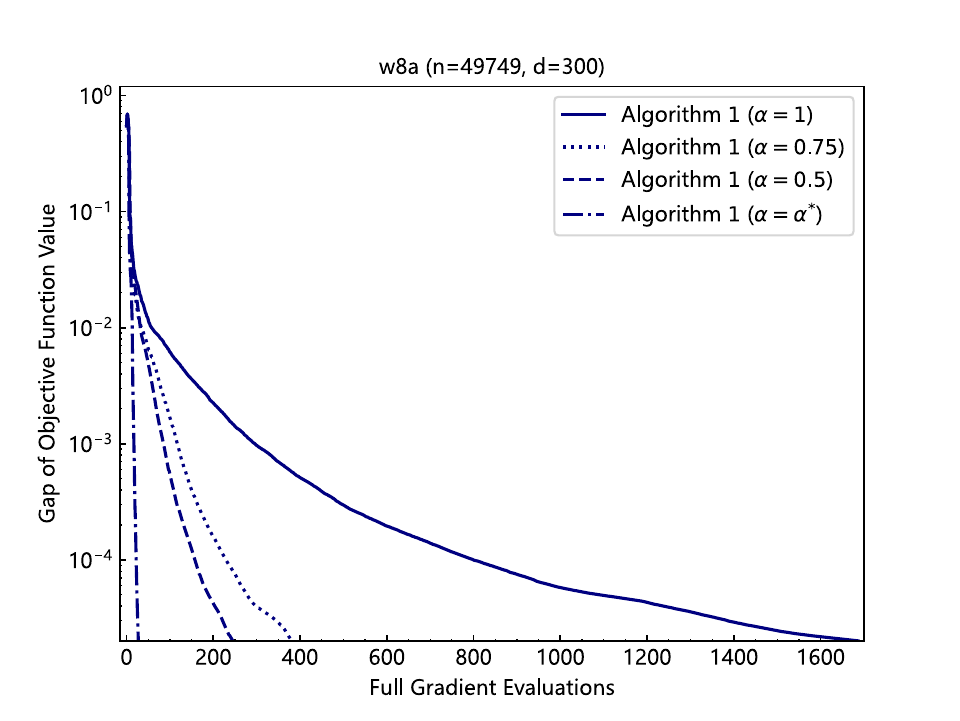}
\caption{w8a}
\label{fig:self_comparison_high_accuracy_w8a}
\end{subfigure}

\caption{Performance comparison of Algorithm $1$ under different $\alpha$ settings}
\label{fig:self-comparison}
\end{figure}

\section{Conclusion}
\label{conclusions}

The trade-off between convergence rate and per-iteration cost, which essentially determine the complexity of an algorithm, is important in (structured) finite-sum optimization. Handling such trade-off typically involves the usage of different algorithmic structure, iteration direction, and other optimization technique. In this work, we propose a unified single-loop variance-reduced acceleration framework. By relating checkpoint update probabilities to momentum parameters and alternating the growth rate of the momentum parameters, the proposed simple framework exhibits a continuous trade-off between convergence rate and per-iteration cost, obtaining a series of new complexity and recovering some known complexity of distinct existing methods as special cases.

The proposed framework allows us to explore the trade-off between acceleration and variance reduction, which is key to achieve the aforementioned rate-cost trade-off and is itself a novel topic in noise-absent finite-sum optimization. The analysis shows that in the proposed framework, acceleration comes with implicit price of heavier variance reduction (achieved by increasing per-iteration cost), which renders the largest allowable acceleration strength not necessarily optimal and highlights the introduced acceleration-variance reduction trade-off in certain scenarios.

The proposed framework also leads to complexity improvement in the family of single-loop variance reduction methods. Under standard assumptions, our method achieves the near-optimal complexity $\Tilde{\mathcal{O}}(n+\sqrt{n}/\sqrt{\epsilon})$, which significantly improves upon previous best complexity $\mathcal{O}(n+n/\sqrt{\epsilon})$ in this family of methods for solving the composite general convex problem.

\clearpage

\appendix

\section{Proof of Lemma \ref{lemma, well-definedness}}
\label{proof of lemma 1}

We begin by proving the following key inequality:
\begin{equation}
\label{lemma 1, equation step 1}
    \xi\left(\alpha_{t+1}^2-\alpha_t^2 \right)\leq \alpha_{t-1}^2-\alpha_t^2+\alpha_t, \ \forall t\in\mathbb{N}^{+}.
\end{equation}
When $\alpha=0$, all $\alpha_t$ are equal and positive, hence (\ref{lemma 1, equation step 1}) obviously holds. The cases where $\alpha\neq 0$ are considered as follows. Consider the case where $t\geq 18$, so that $\alpha_{t-1}, \alpha_t, \alpha_{t+1}$ take the form $a_{\alpha}(t-1)^{\alpha}, a_{\alpha}t^{\alpha}$, and $a_{\alpha}(t+1)^{\alpha}$, respectively. In this case, it suffices to prove
\begin{equation}
\label{lemma 1, equation step 2}
    \alpha_{t+1}^2\leq \alpha_t+\alpha_{t-1}^2, \ \forall t\geq 18,
\end{equation}
since by (\ref{lemma 1, equation step 2}), we can deduce that
\begin{equation*}
    \xi\left(\alpha_{t+1}^2-\alpha_t^2 \right)\leq \alpha_{t+1}^2-\alpha_t^2\leq \alpha_t-\alpha_t^2+\alpha_{t-1}^2,
\end{equation*}
where the first inequality follows from $\alpha_{t+1} \geq \alpha_t$ and the setting $\xi=1/(bc)$ where $b\in [n], c\geq 2$. Define
\begin{equation*}
    g(t)=a_{\alpha}t^{\alpha}+a_{\alpha}^2(t-1)^{2\alpha}-a_{\alpha}^2 (t+1)^{2\alpha},
\end{equation*}
then inequality (\ref{lemma 1, equation step 2}) is equivalent to $g(t)\geq 0$ for $t\geq 18$, which will be proved by monotonicity argument. The simple case $\alpha=1$ will be discussed separately. We prove $g(17)\geq 0$ for $\alpha\in \left(0,1\right)$ as follows. In the $\alpha\in \left(0,1/2\right]$ case, note that
\begin{align*}
    g(17)&= a_{\alpha}\left[17^{\alpha}+a_{\alpha}\left(16^{2\alpha}-18^{2\alpha}\right) \right] \\
    &=a_{\alpha}\left[17^{\alpha}+a_{\alpha}\left(16^{\alpha}-18^{\alpha} \right)\left(16^{\alpha}+18^{\alpha} \right) \right] \\
    &\geq a_{\alpha}\left(16^{\alpha}+18^{\alpha} \right)\left[a_{\alpha}\left(16^{\alpha}-18^{\alpha} \right)+\frac{1}{2} \right] \\
    &\geq \frac{4+\sqrt{2}}{4}(16^{\alpha}+18^{\alpha})\left[\frac{4+\sqrt{2}}{4}\left(\sqrt{16}-\sqrt{18}\right)+\frac{1}{2} \right],
\end{align*}
where in the first inequality the concavity inequality $17^{\alpha}\geq (1/2)\left(16^{\alpha}+18^{\alpha} \right)$ is used, and the second inequality is by the fact that $r_1^{\alpha}-r_2^{\alpha}$ is decreasing w.r.t. $\alpha$, where $1<r_1<r_2$ (this can be shown by taking derivative w.r.t. $\alpha$). Simple calculation shows that $g(17)\geq 0$ in this case. Similarly, in the $\alpha\in \left(1/2,3/4\right]$ case, 
\begin{equation*}
    g(17)\geq \frac{1}{3}\left(16^{\alpha}+18^{\alpha} \right)\left[\frac{1}{3}\left(16^{3/4}-18^{3/4} \right)+\frac{1}{2} \right]\geq 0.
\end{equation*}
In the $\alpha\in\left(3/4,1\right)$ case, 
\begin{equation*}
    g(17)\geq \frac{1}{4}\left(\frac{17}{16}\right)^{\alpha-1}\left(16^{\alpha}+18^{\alpha} \right)\left[\frac{4}{17}\left(\frac{17}{16}\right)^{\alpha}\left(16^{\alpha}-18^{\alpha}\right)+\frac{1}{2} \right].
\end{equation*}
Note that the quantity $17^{\alpha}-\left(308/16\right)^{\alpha}$ has the infimum $-17/8$ on the interval $\left(3/4,1\right)$. This shows that $g(17)\geq 0$. Define 
\begin{equation*}
    g_1(t)=t^{\alpha-1}+2a_{\alpha}(t-1)^{2\alpha-1}-2a_{\alpha}(t+1)^{2\alpha-1},
\end{equation*}
then $g^{\prime}(t)=a_{\alpha}\alpha g_1(t)$ and it suffices to prove $g_1(t)\geq 0$ for $t\geq 17$ to complete the monotonicity argument. 

Consider the case where $\alpha\in(0,1)$. In this case, 
\begin{equation*}
    g_1(t)=t^{\alpha-1}\left( 1+2a_{\alpha}(t-1)^{\alpha}(1-1/t)^{\alpha-1}-2a_{\alpha}(t+1)^{\alpha}(1+1/t)^{\alpha-1}\right).
\end{equation*}
To prove $g_1(t)\geq 0$, it suffices to prove
\begin{equation*}
    \frac{1}{2a_{\alpha}}\geq (t+1)^{\alpha}\left(1+1/t \right)^{\alpha-1}-(t-1)^{\alpha}(1-1/t)^{\alpha-1}.
\end{equation*}
Note that $(1+1/t)^{\alpha-1}\leq (1-1/t)^{\alpha-1}$, it suffices to prove
\begin{equation*}
    1/(2a_{\alpha})\geq (1-1/t)^{\alpha-1}[(t+1)^{\alpha}-(t-1)^{\alpha}].
\end{equation*}
Since $(1-1/t)^{\alpha-1}$ is decreasing w.r.t. $t$, hence it is upper bounded by $(16/17)^{\alpha-1}$ for $t\geq 17$. Define $q_{t}(\alpha)=(t+1)^{\alpha}-(t-1)^{\alpha}$ where $t$ is regarded as a parameter. When $t\geq 17$, taking derivative w.r.t. $\alpha$ gives
\begin{equation*}
    q_t^{\prime}(\alpha)=(t+1)^{\alpha}\log{(t+1)}-(t-1)^{\alpha}\log{(t-1)}.
\end{equation*}
Since $(t+1)^{\alpha}\geq (t-1)^{\alpha}\geq 0$ and $\log{(t+1)}\geq \log{(t-1)}\geq 0$, we have $q_t^{\prime}(\alpha)\geq 0$ for any fixed $t\geq 17$. 

Consider the sub-case where $\alpha\in \left(0,1/2\right]$. Fix any $t\geq 17$, we have
\begin{equation*}
    q_t(\alpha)\leq \lim_{\alpha\rightarrow (1/2)^{-}}q_t(\alpha)=\sqrt{t+1}-\sqrt{t-1}\leq \frac{1}{\sqrt{t-1}}\leq\frac{1}{4}.
\end{equation*}
Hence 
\begin{equation*}
    \left(1-1/t \right)^{\alpha-1}\left[(t+1)^{\alpha}-(t-1)^{\alpha} \right]\leq (16/17)^{\alpha-1}\cdot (1/4)\leq 1/(2a_{\alpha}),
\end{equation*}
where the second inequality is by our choice of $a_{\alpha}=(4+\sqrt{2})/4$, and the inequality desired in this sub-case is proved.

Consider the sub-case where $\alpha\in \left(1/2,3/4\right]$. Fix any $t\geq 17$, we have
\begin{equation*}
    q_t(\alpha)\leq \lim_{\alpha\rightarrow (3/4)^{-}} q_t(\alpha)=(t+1)^{3/4}-(t-1)^{3/4}\leq \frac{3}{4}.
\end{equation*}
Hence
\begin{equation*}
    \left(1-1/t \right)^{\alpha-1}\left[(t+1)^{\alpha}-(t-1)^{\alpha} \right]\leq (16/17)^{\alpha-1}\cdot (3/4)\leq 1/(2a_{\alpha}),
\end{equation*}
where the second inequality is by our choice of $a_{\alpha}=1/3$.

Consider the sub-case where $\alpha\in \left(3/4,1\right)$. Fix any $t\geq 17$, we have $q_t(\alpha)\leq \lim_{\alpha\rightarrow 1^{-}} q_t(\alpha)=2$, so $q_t(\alpha)\leq 2$ for $t\geq 17$. Hence
\begin{equation*}
    (1-1/t)^{\alpha-1}\left[(t+1)^{\alpha}-(t-1)^{\alpha} \right]\leq (16/17)^{\alpha-1}\cdot 2\leq 1/(2a_{\alpha}),
\end{equation*}
where the second inequality is by our choice of $a_{\alpha}=(1/4)\cdot(17/16)^{\alpha-1}$.

Consider the case where $\alpha=1$. In this case, $g_1(t)=1-4a_{\alpha}\geq 0$ since $a_{\alpha}|_{\alpha=1}=1/4$.

Altogether, by monotonicity, we have $g(t)\geq 0$ for $t\geq 18, \alpha\in \left(0,1\right]$. Thus (\ref{lemma 1, equation step 2}) is proved for all $\alpha\in \left(0,1\right]$, which implies that (\ref{lemma 1, equation step 1}) holds for all $t\geq 18$, $\alpha\in \left(0,1\right]$. We then prove (\ref{lemma 1, equation step 1}) for $1\leq t\leq 17$ under all $\alpha\in\left(0,1\right]$ settings (except for $\alpha=0$). When $t=17$, since $\alpha_{18}^2\geq \alpha_{17}^2$ for $\alpha\in \left(0,1\right]$, by the reasoning at the beginning of this proof, it suffices to prove $\alpha_{18}^2\leq \alpha_{17}+\alpha_{16}^2$. In the $\alpha\in \left(0,1/2\right]$ sub-case, $\alpha_{18}^2=(1+\sqrt{2}/4)^2\cdot 18^{2\alpha}\leq 36=\alpha_{16}^2 $, in the $\alpha\in\left(1/2,3/4\right]$ sub-case, $\alpha_{18}^2=(1/9)\cdot 18^{2\alpha}\leq 36$, and in the $\alpha\in \left(3/4,1\right]$ sub-case, $\alpha_{18}^2=\left(1/16\right)\cdot \left(17/16\right)^{2\alpha-2}18^{2\alpha}\leq 36=\alpha_{16}^2$, this concludes the $t=17$ case for all $\alpha\in \left(0,1\right]$. When $t=16$, $\alpha_{17}^2-\alpha_{16}^2\leq \alpha_{18}^2-\alpha_{16}^2\leq 0$, while the RHS of (\ref{lemma 1, equation step 1}) equals $\alpha_{16}\geq 0$, hence the $t=16$ case is proved. Note that (\ref{lemma 1, equation step 1}) obviously holds for $1\leq t\leq 15$, since the left hand side (LHS) equals $0$ and the RHS is positive. 

Altogether, the key inequality (\ref{lemma 1, equation step 1}) is proved for all $\alpha\in[0,1]$, and its corollaries are listed as follows. First, it follows from (\ref{lemma 1, equation step 1}) that the following inequality holds:
\begin{equation}
\label{lemma 1, equation step 3}
    \alpha_{t-1}^2-\alpha_t^2+\alpha_t\geq 0, \ \forall t\in \mathbb{N}^{+}.
\end{equation}
Since when $t\in \mathbb{N}^{+}\setminus \{16\}$, we have $\alpha_{t+1}^2\geq \alpha_{t}^2$, by which one concludes $\alpha_{t-1}^2-\alpha_t^2+\alpha_t\geq 0$, where (\ref{lemma 1, equation step 1}) is used. And when $t=16$, $\alpha_{15}^2-\alpha_{16}^2+\alpha_{16}=\alpha_{16}\geq 0$. Second, it follows from (\ref{lemma 1, equation step 1}) that the following inequality holds:
\begin{equation}
\label{lemma 1, equation step 4}
    \Tilde{\alpha}_0+\alpha_0^2-\alpha_{t-1}^2+\sum_{j=1}^{t-1}\alpha_j \geq \xi \alpha_t^2\geq 0, \ \forall t\in\mathbb{N}^{+}.
\end{equation}
We show this by mathematical induction. When $t=1$, $\Tilde{\alpha}_0\geq \xi\alpha_1^2$ by the setting of $\Tilde{\alpha}_0$. Suppose (\ref{lemma 1, equation step 4}) holds for some $t\in\mathbb{N}^{+}$. Then
\begin{align*}
    \Tilde{\alpha}_0+\alpha_0^2-\alpha_t^2+\sum_{j=1}^t \alpha_j-\xi\alpha_{t+1}^2 &=\Tilde{\alpha}_0+\alpha_0^2-\alpha_{t-1}^2+\sum_{j=1}^{t-1}\alpha_j+\alpha_{t-1}^2-\alpha_t^2+\alpha_t-\xi \alpha_{t+1}^2 \\
    &\geq \xi\left(\alpha_t^2-\alpha_{t+1}^2 \right)+\alpha_{t-1}^2-\alpha_t^2+\alpha_t\geq 0,
\end{align*}
where in the first inequality we use the induction hypothesis and in the second inequality (\ref{lemma 1, equation step 1}) is used. By mathematical induction, (\ref{lemma 1, equation step 4}) is proved. Recalling the definition of the $\{p_t\}$ sequence in Algorithm \ref{acceleration framework}, $p_t$ can be reformulated as follows
\begin{equation}
\label{lemma 1, equation step 5}
    p_t=\frac{\alpha_{t-1}^2-\alpha_t^2+\alpha_t+\xi\alpha_t^2}{\alpha_{t-1}^2-\alpha_t^2+\alpha_t+\Tilde{\alpha}_0+\alpha_0^2-\alpha_{t-1}^2+\sum_{j=1}^{t-1}\alpha_j}.
\end{equation}
By (\ref{lemma 1, equation step 3}), (\ref{lemma 1, equation step 4}) and (\ref{lemma 1, equation step 5}), we have
\begin{equation}
\label{lemma 1, equation step 6}
    0\leq p_t\leq 1, \ \forall t\in\mathbb{N}^{+}.
\end{equation}
One easily checks that $\alpha_{17}>1$ for $\alpha=0, \alpha\in\left(0,1/2\right], \alpha\in \left(1/2,3/4\right]$, and $\alpha\in\left(3/4,1\right]$, respectively. In fact, $\alpha_{17}$ allows a uniform deterministic constant lower bound for all $\alpha\in[0,1]$. Specifically, $\alpha_{17}=6$ for $\alpha=0$, $\alpha_{17}\geq 1+\sqrt{2}/4$ for $\alpha\in \left(0,1/2\right]$, $\alpha_{17}\geq 4/3$ for $\alpha\in\left(1/2,3/4\right]$, and $\alpha_{17}\geq 32/17$ for $\alpha\in\left(3/4,1\right]$. By the monotonicity of the $\alpha_t$ sequence w.r.t. $t$, and noting that $\alpha_{t}=6$ for $1\leq t\leq 16$, we have $\alpha_t>1$ for $\alpha\in [0,1]$, which shows that $\tau_t=1/\alpha_t\in(0,1)$. Note that $\tau_t\leq \max \left\{1/6,\alpha_{17}^{-1} \right\}<1$, by choosing $c= \max \left\{2, b^{-1}\max \left\{6/5,\left(1-\alpha_{17}^{-1}\right)^{-1} \right\}\right\}+1$, we have $1-\tau_t-\xi\in (0,1)$. The fact that $\alpha_{17}$ is uniformly bounded away from $1$ guarantees that the constant in $c$ does not explode for any choice of $\alpha\in[0,1]$. Since $b\geq 1$ and $\alpha_{17}\geq 4/3$ for all $\alpha\in [0,1]$, there exists a constant $C>0$ that does not depend on any problem-dependent constant (e.g., $n$) such that\footnote{For example, one can take $C=5$.}
\begin{equation}
\label{lemma 1, equation step 7}
    c\leq C, \ \forall \alpha\in[0,1].
\end{equation}
And by the choice of $c$, we also have $\xi\in(0,1)$. Altogether, the following inequalities are ensured
\begin{equation}
\label{lemma 1, equation step 8}
    \tau_t\in(0,1), \ \xi\in(0,1), \ (1-\xi-\tau_t)\in(0,1), \ \forall t\in\mathbb{N}^{+}.
\end{equation}
The claimed inequalities in Lemma \ref{lemma, well-definedness}, namely (\ref{lemma 1, equation step 4}), (\ref{lemma 1, equation step 6}), (\ref{lemma 1, equation step 7}), and (\ref{lemma 1, equation step 8}) are proved. They ensure the well-definedness of Algorithm \ref{acceleration framework} and are fundamental to its convergence proof. $\hfill\qedsymbol$

\section{Bounding the Variance of the SVRG Estimator}
\label{proof of SVRG variance upper bound}

The following lemma mimics Lemma $2.4$ of \cite{allen2018katyusha}.

\begin{lemma}
\label{lemma, variance control}
In Algorithm \ref{acceleration framework}, the variance of the mini-batch version of the SVRG estimator $\Tilde{g}_{t+1}$ can be bounded as follows:
    \begin{equation*}
        \mathbb{E}_{J_t}\left[\Vert \Tilde{g}_{t+1}-\nabla f(x_{t+1})\Vert^2 \right]\leq \frac{2L}{b}\left(f(w_t)-f(x_{t+1})-\langle \nabla f(x_{t+1}),w_t-x_{t+1}\rangle \right).
    \end{equation*}
\end{lemma}
\begin{proof}
For a given iteration $t\in \mathbb{N}^{+}$, define $X_j=\nabla f_j(x_{t+1})-\nabla f_j(w_t)$. Then
\begin{equation}
\label{lemma 2, equation 1}
    \mathbb{E}_{J_t}\sum_{j\in J_t}X_j=\sum_{\text{all possible} \  J_t }\frac{1}{C_n^b}\sum_{j\in J_t} X_j=\frac{1}{C_n^b}\sum_{j=1}^n C_{n-1}^{b-1}X_j=\frac{b}{n}\sum_{j=1}^n X_j,
\end{equation}
where $C_n^m=n!/\left[m!(n-m)! \right]$. Hence
\begin{align*}
    \mathbb{E}_{J_t}\Vert \Tilde{g}_{t+1}-\nabla f(x_{t+1})\Vert^2 &=\mathbb{E}_{J_t}\Vert \frac{1}{b}\sum_{j\in J_t} X_j+\nabla f(w_t)-\nabla f(x_{t+1})\Vert^2 \\
    &=\mathbb{E}_{J_t} \Vert \frac{1}{b}\sum_{j\in J_t}X_j-\mathbb{E}_{J_t}\frac{1}{b}\sum_{j\in J_t}X_j\Vert^2 \\
    &\leq \frac{1}{b^2}\mathbb{E}_{J_t}\Vert \sum_{j\in J_t} X_j\Vert^2 \\
    &=\frac{1}{b n}\sum_{j=1}^n \Vert \nabla f_j(x_{t+1})-\nabla f_j(w_t)\Vert^2 \\
    &\leq \frac{2L}{b}\left(f(w_t)-f(x_{t+1})-\langle \nabla f(x_{t+1}),w_t-x_{t+1}\rangle\right),
\end{align*}
where in the second equality we apply (\ref{lemma 2, equation 1}), the third equality can be derived in a way similarly to (\ref{lemma 2, equation 1}) and we refer to Lemma $18$ and Lemma $19$ of \cite{driggs2022accelerating}, in the last inequality the following property is applied:
\begin{equation*}
    \frac{1}{2L}\Vert \nabla g(x)-\nabla g(y)\Vert^2\leq g(x)-g(y)-\langle \nabla g(y),x-y\rangle, \ \forall x,y\in\mathbb{R}^d,
\end{equation*}
where $g:\mathbb{R}^d\rightarrow \mathbb{R}$ is a L-smooth convex function, and this property is a standard tool which can be found in \cite{nesterov2018lectures}.
\end{proof}

\section{Proof of Theorem \ref{theorem, convergence}}
\label{proof of theorem 1}

Note that 
\begin{align}
    \alpha_t \left(f(x_{t+1})-f(x^*)\right)&\leq \alpha_t\langle \nabla f(x_{t+1}),x_{t+1}-x^*\rangle  \nonumber \\
    &=\alpha_t \langle \nabla f(x_{t+1}),z_t-x^*\rangle+\alpha_t \langle \nabla f(x_{t+1}),x_{t+1}-z_t\rangle \nonumber \\
    &=\mathbb{E}_{J_t}\alpha_t \langle \Tilde{g}_{t+1},z_t-x^*\rangle+\alpha_t\langle \nabla f(x_{t+1}),x_{t+1}-z_t\rangle \nonumber \\
    &= \mathbb{E}_{J_t}\alpha_t\langle \Tilde{g}_{t+1},z_t-z_{t+1}\rangle+\mathbb{E}_{J_t}\alpha_t\langle \Tilde{g}_{t+1},z_{t+1}-x^*\rangle \nonumber \\
    &+\alpha_t\langle \nabla f(x_{t+1}),x_{t+1}-z_t\rangle,
\label{theorem 1, equation 1}
\end{align}
where in the first inequality we apply the convexity of $f$, and the second equality is by the unbiasedness of the estimator $\Tilde{g}_{t+1}$. The three terms in the RHS of (\ref{theorem 1, equation 1}) are separately bounded as follows. For the first term,
\begin{align}
    \alpha_t\langle \Tilde{g}_{t+1},z_t-z_{t+1}\rangle &=\alpha_t\langle \nabla f(x_{t+1}),z_t-z_{t+1}\rangle+\alpha_t\langle \Tilde{g}_{t+1}-\nabla f(x_{t+1}),z_t-z_{t+1}\rangle \nonumber \\
    &=\alpha_t^2 \langle \nabla f(x_{t+1}),x_{t+1}-y_{t+1}\rangle+\alpha_t^2\langle \Tilde{g}_{t+1}-\nabla f(x_{t+1}),x_{t+1}-y_{t+1}\rangle \nonumber \\
    &\leq \alpha_t^2\left[f(x_{t+1})-f(y_{t+1})+\frac{L}{2}\Vert x_{t+1}-y_{t+1}\Vert^2 \right]+\frac{\alpha_t^2}{2L c}\Vert \Tilde{g}_{t+1}-\nabla f(x_{t+1})\Vert^2 \nonumber \\
    &+\frac{\alpha_t^2 L c}{2}\Vert x_{t+1}-y_{t+1}\Vert^2 \nonumber \\
    &\leq \alpha_t^2 \left[f(x_{t+1})-f(y_{t+1})+\frac{L}{2}\Vert x_{t+1}-y_{t+1}\Vert^2 \right]+\frac{\alpha_t^2 L c}{2}\Vert x_{t+1}-y_{t+1}\Vert^2 \nonumber \\
    &+\frac{\alpha_t^2 }{b c } \left[f(w_t)-f(x_{t+1})-\langle \nabla f(x_{t+1}),w_t-x_{t+1}\rangle \right],
    \label{theorem 1, equation 2}
\end{align}
where in the second equality the relation $z_t-z_{t+1}=\alpha_t (x_{t+1}-y_{t+1})$ is used, in the first inequality the $L$-smoothness of $f$ and the Young inequality are used, and the last inequality is by Lemma \ref{lemma, variance control}. For the term $\alpha_t\langle \Tilde{g}_{t+1},z_{t+1}-x^*\rangle$, we have
\begin{align}
    \alpha_t\langle \Tilde{g}_{t+1},z_{t+1}-x^*\rangle &=\langle \frac{1}{\eta} (z_t-z_{t+1} )-\alpha_t \bar{g}_{t+1}, z_{t+1}-x^*\rangle \nonumber \\
    &=\frac{1}{\eta}\langle z_t-z_{t+1},z_{t+1}-x^*\rangle-\alpha_t \langle \bar{g}_{t+1},z_{t+1}-x^*\rangle \nonumber \\
    &\leq \frac{1}{2\eta}\left(\Vert z_t-x^*\Vert^2-\Vert z_{t+1}-x^*\Vert^2-\Vert z_t-z_{t+1} \Vert^2\right) \nonumber \\
    &+\alpha_t\left(l(x^*)-l(z_{t+1}) \right),
    \label{theorem 1, equation 3}
\end{align}
where $\bar{g}_{t+1}$ is a subgradient of $l$ at $z_{t+1}$, and in the first equality we use the equivalent update formula\footnote{This is derived from the optimality of $z_{t+1}$ in the optimization problem
\begin{equation*}
    z_{t+1}=\arg\min_{z\in\mathbb{R}^d} \left\{\frac{1}{2\alpha_t\eta}\Vert z-z_t\Vert^2+\langle \Tilde{g}_{t+1},z\rangle+l(z) \right\},
\end{equation*} which is a standard result and can be found in \cite{doi:10.1137/1.9781611974997}.} of $z_{t+1}$, i.e., $z_{t+1}=z_{t}-\alpha_{t}\eta\left(\Tilde{g}_{t+1}+\bar{g}_{t+1} \right)$, in the last inequality we use the convexity of $l$ and the three-point property $\langle a-b,b-c\rangle=(1/2)\cdot (\Vert a-c\Vert^2-\Vert a-b\Vert^2-\Vert b-c\Vert^2), \ \forall a,b,c\in\mathbb{R}^d$. Recall  $x_{t+1}=\tau_t z_t+\xi w_t+(1-\xi-\tau_t)y_t$ in Algorithm \ref{acceleration framework} and the setting $\tau_t\alpha_t=1$, we have
\begin{equation}
\label{theorem 1, equation 4}
    x_{t+1}-z_t=\alpha_t\xi (w_t-x_{t+1})+\alpha_t (1-\xi-\tau_t)(y_t-x_{t+1}).
\end{equation}
Then for the term $\alpha_t \langle \nabla f(x_{t+1}),x_{t+1}-z_t\rangle$, we have
\begin{align}
    & \quad \  \alpha_t \langle \nabla f(x_{t+1}),x_{t+1}-z_t\rangle \nonumber \\
    &=\alpha_t^2\xi \langle \nabla f(x_{t+1}),w_t-x_{t+1}\rangle+\alpha_t^2 (1-\xi-\tau_t)\langle \nabla f(x_{t+1}),y_t-x_{t+1}\rangle \nonumber \\
    &\leq \alpha_t^2\xi\langle \nabla f(x_{t+1}),w_t-x_{t+1}\rangle+\alpha_t^2(1-\xi-\tau_t)(f(y_t)-f(x_{t+1})),
    \label{theorem 1, equation 5}
\end{align}
where the first equality is by (\ref{theorem 1, equation 4}), the last inequality is by the convexity of $f$ and Lemma \ref{lemma, well-definedness} which ensures $1-\xi-\tau_t\geq 0$. Taking expectation and combining (\ref{theorem 1, equation 1}), (\ref{theorem 1, equation 2}), (\ref{theorem 1, equation 3}), and (\ref{theorem 1, equation 5}), we have
\begin{align}
    & \quad \ \alpha_t\mathbb{E}\left(f(x_{t+1})-f(x^*)) \right) \nonumber \\
    &\leq \alpha_t^2 \mathbb{E}\left[f(x_{t+1})-f(y_{t+1})+\frac{L}{2}\Vert x_{t+1}-y_{t+1}\Vert^2 \right] 
    +\frac{\alpha_t^2 L c}{2}\mathbb{E}\Vert x_{t+1}-y_{t+1}\Vert^2 \nonumber \\
    &+\alpha_t^2 \xi\mathbb{E}\left[f(w_t)-f(x_{t+1})-\langle \nabla f(x_{t+1}),w_t-x_{t+1}\rangle \right] \nonumber \\
    &+\frac{1}{2\eta}\mathbb{E}\left(\Vert z_t-x^*\Vert^2-\Vert z_{t+1}-x^*\Vert^2-\Vert z_t-z_{t+1}\Vert^2 \right) +\alpha_t\mathbb{E}\left(l(x^*)-l(z_{t+1}) \right) \nonumber \\
    &+\alpha_t^2\xi\mathbb{E}\langle \nabla f(x_{t+1}),w_t-x_{t+1}\rangle+\alpha_t^2(1-\xi-\tau_t)\mathbb{E}\left(f(y_t)-f(x_{t+1})\right).
    \label{theorem 1, equation 6}
\end{align}
Note again that $z_{t+1}-z_t=\alpha_t(y_{t+1}-x_{t+1})$ by the $y_{t+1}$ update in Algorithm \ref{acceleration framework}, inequality (\ref{theorem 1, equation 6}) becomes
\begin{align}
    \alpha_t^2\mathbb{E}\left(f(y_{t+1})-F(x^*) \right)&\leq \alpha_t^2(1-\xi-\tau_t)\mathbb{E}\left(f(y_t)-F(x^*) \right)+\alpha_t^2\xi\mathbb{E}\left(f(w_t)-F(x^*) \right) \nonumber \\
    &-\alpha_t\mathbb{E}l(z_{t+1})+\alpha_t^2\left(\frac{L}{2}+\frac{L c}{2}-\frac{1}{2\eta}\right)\mathbb{E}\Vert x_{t+1}-y_{t+1}\Vert^2 \nonumber \\
    &+\frac{1}{2\eta}\mathbb{E}\Vert z_t-x^*\Vert^2-\frac{1}{2\eta}\mathbb{E}\Vert z_{t+1}-x^*\Vert^2.
    \label{theorem 1, equation 7}
\end{align}
By update formulae of $x_{t+1}$ and $y_{t+1}$ in Algorithm \ref{acceleration framework}, we have
\begin{equation*}
    y_{t+1}=\tau_t z_{t+1}+\xi w_t+(1-\xi-\tau_t)y_t,
\end{equation*}
it follows (by Lemma \ref{lemma, well-definedness}, $\tau_t$, $\xi$, and $(1-\xi-\tau_t)$ all belong to $(0,1)$)
\begin{equation}
\label{theorem 1, equation 8}
    l(y_{t+1})\leq \tau_t l(z_{t+1})+\xi l(w_t)+(1-\xi-\tau_t)l(y_t).
\end{equation}
Note that $L/2+L c/2-1/(2\eta)\leq 0$ since $\eta\leq 1/(c L+L)$. Then combining (\ref{theorem 1, equation 7}) and (\ref{theorem 1, equation 8}), we arrive at
\begin{align}
    \alpha_t^2 \mathbb{E}\left(F(y_{t+1})-F(x^*) \right)&\leq \alpha_t^2(1-\xi-\tau_t)\mathbb{E}\left(F(y_t)-F(x^*) \right)+\alpha_t^2\xi\mathbb{E}\left(F(w_t)-F(x^*) \right) \nonumber \\
    &+\frac{1}{2\eta}\mathbb{E}\Vert z_t-x^*\Vert^2-\frac{1}{2\eta}\mathbb{E}\Vert z_{t+1}-x^*\Vert^2.
    \label{theorem 1, equation 9}
\end{align}
By the tower property of conditional expectation and the probabilistic update formula of $w_{t+1}$ in Algorithm \ref{acceleration framework}, we have
\begin{equation}
    \mathbb{E}\left(F(w_{t+1})-F(x^*)\right)\leq (1-p_t) \mathbb{E}\left(F(w_t)-F(x^*) \right)+p_t\mathbb{E}\left(F(y_t)-F(x^*) \right).
    \label{theorem 1, equation 10}
\end{equation}
For $t\in \mathbb{N}$, define
\begin{align*}
    \mathcal{L}_{t+1} =&\alpha_t^2 \mathbb{E}\left(F(y_{t+1})-F(x^*) \right)+\left(\Tilde{\alpha}_0+\alpha_0^2-\alpha_t^2+\sum_{j=1}^t \alpha_j \right)\mathbb{E}\left(F(w_{t+1})-F(x^*) \right) \\
    &+\frac{1}{2\eta}\mathbb{E}\Vert z_{t+1}-x^*\Vert^2.
\end{align*}
Multiplying (\ref{theorem 1, equation 10}) by $\Tilde{\alpha}_0+\alpha_0^2-\alpha_t^2+\sum_{j=1}^t \alpha_j$ and combining the resultant inequality with (\ref{theorem 1, equation 9}), we obtain (note that $\Tilde{\alpha}_0+\alpha_0^2-\alpha_t^2+\sum_{j=1}^t \alpha_j\geq 0$ by Lemma \ref{lemma, well-definedness})
\begin{equation}
    \mathcal{L}_{t+1}\leq \mathcal{L}_t, \ \forall t\in\mathbb{N},
    \label{theorem 1, equation 11}
\end{equation}
where the inequality follows from our choice of $\tau_t$ and $p_t$, i.e., one can rewrite the coefficients with solely the $\{\alpha_t\}_{t\in\mathbb{N}}$ sequence to see that the inequality holds. Using (\ref{theorem 1, equation 11}) repeatedly, the theorem is then proved. $\hfill\qedsymbol$

\section{Proof of Theorem \ref{theorem, complexity}}
\label{proof of theorem 2}

We divide the proof into cases where the $\{\alpha_t\}$ sequence has different growing rate, which is determined by $\alpha$ in Lemma \ref{lemma, well-definedness}. In all cases, we always consider the convergence of the quantity $\mathbb{E}\left(F(w_t)-F(x^*)\right)$.

In the case where $\alpha=0$, we have $\alpha_t\equiv 6$ for all $t\in\mathbb{N}$. In this case, we have
\begin{equation*}
    p_t=\frac{36\xi+6}{\Tilde{\alpha}_0+6t}.
\end{equation*}
By (\ref{theorem 1, equation 0}), Algorithm \ref{acceleration framework} needs $T=\mathcal{O}(1/\epsilon)$ iterations to reach an $\epsilon$-accurate solution of (\ref{problem setting}). In expectation, Algorithm \ref{acceleration framework} calls IFO $\mathcal{O}(b+np_t)$ times per iteration. Hence the IFO complexity in expectation $\mathbb{E}C_{\text{comp}}$ is $\mathcal{O}(b/\epsilon+n\sum_{t=1}^{\lceil 1/\epsilon \rceil}p_t)$, where the dominant term can be further bounded as follows
\begin{align*}
    \frac{b}{\epsilon}+n\sum_{t=1}^{\lceil 1/ \epsilon \rceil}p_t &\leq \frac{b}{\epsilon}+\frac{n\left(36\xi+6\right)}{6}\sum_{t=1}^{\lceil 1/\epsilon \rceil} \frac{1}{t} \\
    &=\frac{b}{\epsilon}+n\left(6\xi+1 \right)+n\left(6\xi+1\right)\sum_{t=2}^{\lceil 1/\epsilon \rceil} \frac{1}{t} \\
    &\leq \frac{b}{\epsilon}+7n+7n\log\left(\lceil 1/\epsilon \rceil\right),
\end{align*}
where the last inequality is by $\xi\in(0,1)$ and $\sum_{t=2}^{\lceil 1/\epsilon \rceil}(1/t)\leq \int_{1}^{\lceil 1/\epsilon \rceil}(1/t)\text{d}t$. In addition, when $n< \lceil 1/\epsilon\rceil$, 
\begin{align*}
    \frac{b}{\epsilon}+n\sum_{t=1}^{\lceil 1/\epsilon\rceil}p_t &\leq \frac{b}{\epsilon}+\frac{n(36\xi+6)}{6}\sum_{t=1}^{n}\frac{1}{t}+\frac{n(36\xi+6)}{6}\sum_{t=n+1}^{\lceil 1/\epsilon\rceil}\frac{1}{t} \\
    &\leq \frac{b}{\epsilon}+7n+7n\log(n)+7n\sum_{t=n+1}^{\lceil 1/\epsilon\rceil}\frac{1}{n} \\
    &\leq \frac{b}{\epsilon}+7n\log(n)+7\lceil 1/\epsilon\rceil.
\end{align*}
In conclusion, in the $\alpha=0$ case, the optimal choice of the mini-batch size is $b^{*}=1$, and the IFO complexity in expectation is $\mathcal{O}\left(n\min \{\log(1/\epsilon),\log(n)\}+1/\epsilon \right)$.

In the case where $\alpha\in (0,1]$, we have
\begin{align*}
    \Tilde{\alpha}_0+\alpha_0^2-\alpha_{t}^2+\sum_{j=1}^{t}\alpha_j &=\Tilde{\alpha}_0+132-a_{\alpha}^2t^{2\alpha}+a_{\alpha}\sum_{j=17}^t j^{\alpha} \\
    &\geq \Tilde{\alpha}_0+132-\frac{a_{\alpha}}{\alpha+1}16^{\alpha+1}-a_{\alpha}^2t^{2\alpha}+\frac{a_{\alpha}}{\alpha+1}t^{\alpha+1},
\end{align*}
where the inequality is by $\sum_{j=17}^t j^{\alpha}\geq \int_{16}^t u^{\alpha}\text{d}u$. Note that $132\geq \left(a_{\alpha}16^{\alpha+1} \right)/(\alpha+1)$ for all $\alpha\in \left(0,1\right]$. For $\alpha\in(0,1)$, define
\begin{equation*}
    T_{\alpha}=\left[2a_{\alpha}(\alpha+1)\right]^{\frac{1}{1-\alpha}},
\end{equation*}
then we have $a_{\alpha}^2t^{2\alpha}\leq \left(a_{\alpha}t^{\alpha+1}\right)/(2\alpha+2)$ for all $t\in \mathbb{N}^{+}\cap \left[T_{\alpha},+\infty \right)$. One easily checks that $T_{\alpha}\leq 17$ for all $\alpha\in (0,1)$. Moreover, when $\alpha=1$, we have $a_{\alpha}|_{\alpha=1}=1/4$, and $-a_{\alpha}^2 t^{2\alpha}+\left(a_{\alpha}t^{\alpha+1}\right)/(\alpha+1)=\frac{t^2}{16}$. For $\alpha\in\left(0,1\right]$, define
\begin{equation*}
    \Tilde{a}_{\alpha} = \begin{cases}
    \frac{a_{\alpha}}{2\alpha+2} & \text{when} \ \alpha\in(0,1) \\
    \frac{1}{16} & \text{when} \ \alpha=1
    \end{cases}
    ,
\end{equation*}
then for any $\alpha\in\left(0,1\right]$, when $t\geq 17$, we have $\Tilde{\alpha}_0+\alpha_0^2-\alpha_{t}^2+\sum_{j=1}^{t}\alpha_j\geq \Tilde{a}_{\alpha}t^{\alpha+1}$. Since $T_{\alpha}$ admits a uniform constant upper bound $17$ for any $\alpha\in \left(0,1\right]$, one can find a $\Tilde{\Tilde{a}}_{\alpha}>0$ (which is free from problem dependent parameters and possibly smaller than $\Tilde{a}_{\alpha}$) so that $\Tilde{\alpha}_0+\alpha_0^2-\alpha_{t}^2+\sum_{j=1}^{t}\alpha_j\geq \Tilde{\Tilde{a}}_{\alpha}t^{\alpha+1}$ for all $t\in\mathbb{N}^{+}$.  Hence by Theorem \ref{theorem, convergence}, Algorithm \ref{acceleration framework} needs $T=\mathcal{O}\left(1/\epsilon^{1/\left(\alpha+1\right)}\right)$ iterations to reach an $\epsilon$-accurate solution of (\ref{problem setting}). To estimate
\begin{equation}
\label{theorem 2, equation 1}
    \mathbb{E}C_{\text{comp}}=b/\epsilon^{1/\left(\alpha+1\right)}+n\sum_{t=1}^{\lceil1/\epsilon^{ 1/\left(\alpha+1\right)}\rceil}p_t,
\end{equation}
the update probability $p_t$ is bounded as follows 
\begin{align}
    p_t &\leq \Tilde{a}_{\alpha}^{-1}\frac{a_{\alpha}^2(t-1)^{2\alpha}-(1-\xi)a_{\alpha}^2t^{2\alpha}+a_{\alpha}t^{\alpha}}{t^{\alpha+1}} \nonumber \\
    &=\Tilde{a}_{\alpha}^{-1} \frac{a_{\alpha}^2\left[(t-1)^{2\alpha}-t^{2\alpha} \right]+a_{\alpha}^2\xi t^{2\alpha}+a_{\alpha}t^{\alpha}}{t^{\alpha+1}},
    \label{theorem 2, equation 2}
\end{align}
where $t\geq 17$ and $\xi=1/(b c)$ by our setting. Note that although for $t<17$ the corresponding $p_t$ is not upper bounded carefully, we can always upper bound it by $1$, which at most leads to an addend being constant multiple of $n$ in the complexity, i,e., $\mathcal{O}(n)$. For $t\geq 17$, the inequality (\ref{theorem 2, equation 1}) implies
\begin{equation}
\label{theorem 2, equation 3}
    p_t\leq \Tilde{a}_{\alpha}^{-1}a_{\alpha}^2\xi \frac{1}{t^{1-\alpha}}+\Tilde{a}_{\alpha}^{-1}a_{\alpha}\frac{1}{t}.
\end{equation}
Note that 
\begin{align}
    Q_1(\epsilon,\alpha)=\sum_{t=17}^{\lceil 1/\epsilon^{1/\left(\alpha+1\right)}\rceil} \frac{1}{t^{1-\alpha}} &\leq \int_{16}^{\lceil 1/\epsilon^{1/\left(\alpha+1\right)}\rceil}\frac{1}{t^{1-\alpha}}\text{d}t=\mathcal{O}\left( \frac{1}{\alpha}\frac{1}{\epsilon^{\frac{\alpha}{\alpha+1}}}\right), \label{theorem 2, equation 4} \\ Q_2(\epsilon,\alpha)&=\sum_{t=17}^{\lceil 1/\epsilon^{1/\left(\alpha+1\right)}\rceil }\frac{1}{t}=\mathcal{O}\left( \frac{\log(1/\epsilon)}{\alpha+1}\right). \label{theorem 2, equation 5}
\end{align}
The quantity $Q_1(\epsilon,\alpha)$ is essentially upper bounded by $(1/\alpha)\cdot (1/\epsilon^{(\alpha/(\alpha+1))})$, which explodes when $\alpha\rightarrow 0^{+}$, hence the $\alpha\in\left(0,1/2\right]$ case is further handled as follows. Without loss of generality, assume $\epsilon\in (0,1)$, then for $\alpha\in \left(0,\min\left\{\hat{\alpha},1/10\right\}\right]$, when $\hat{\alpha}=\log(2)/\log(\lceil 1/\epsilon\rceil)$, i.e., $n>\lceil 1/\epsilon\rceil$,
\begin{equation*}
    Q_1(\epsilon,\alpha)=\sum_{t=17}^{\lceil 1/\epsilon^{1/(\alpha+1)}\rceil} \frac{t^{\alpha}}{t}\leq \sum_{t=17}^{\lceil 1/\epsilon^{(1/(\alpha+1))} \rceil}\frac{\left(\lceil 1/\epsilon \rceil\right)^{\hat{\alpha}}}{t}=\mathcal{O}(\log(1/\epsilon)).
\end{equation*}
When $\hat{\alpha}=\log(2)/\log(n)$, i.e., $n<\lceil 1/\epsilon\rceil$, if $\lceil 1/\epsilon^{1/(\alpha+1)}\rceil\leq n<\lceil 1/\epsilon\rceil$,
\begin{equation*}
    Q_1(\epsilon,\alpha)\leq \sum_{t=17}^{n}\frac{t^{\alpha}}{t}\leq \sum_{t=17}^{n} \frac{n^{\hat{\alpha}}}{t}=\mathcal{O}(\log(n)),
\end{equation*}
if $n<\lceil 1/\epsilon^{1/(\alpha+1)}\rceil$,
\begin{align*}
    Q_1(\epsilon,\alpha) &\leq \sum_{t=17}^{\lceil 1/\epsilon^{1/(\alpha+1)}\rceil}\frac{1}{t^{1-\alpha}}=\sum_{t=17}^{n}\frac{1}{t^{1-\alpha}}+\sum_{t=n+1}^{\lceil 1/\epsilon^{1/(\alpha+1)}\rceil}\frac{1}{n^{1-\alpha}} \\
    &\leq \sum_{t=17}^{n}\frac{n^{\alpha}}{t}+\sum_{t=n+1}^{\lceil 1/\epsilon^{1/(\alpha+1)}\rceil}\frac{n^{\alpha}}{n}\leq \sum_{t=17}^{n}\frac{n^{\hat{\alpha}}}{t}+\sum_{t=n+1}^{\lceil 1/\epsilon^{1/(\alpha+1)}\rceil}\frac{n^{\hat{\alpha}}}{n} \\
    &=\mathcal{O}\left(\log(n)+\frac{1}{n}\frac{1}{\epsilon^{1/(\alpha+1)}} \right).
\end{align*}
Hence for $\alpha\in \left(0,\min\left\{\hat{\alpha},1/10\right\}\right]$, 
\begin{equation}
\label{theorem 2, equation 6}
    Q_1(\epsilon,\alpha)=\mathcal{O}\left(\left(1+\frac{1}{n}\frac{1}{\epsilon^{1/(\alpha+1)}}\right)\min\left\{\log(n),\log(1/\epsilon) \right\}\right).
\end{equation}
For $\alpha\in \left(\min\{\hat{\alpha},1/10\},1 \right]$, we have $1/\alpha=\mathcal{O}(\min\{\log(n),\log(1/\epsilon)\})$, then by (\ref{theorem 2, equation 4}), 
\begin{equation}
\label{theorem 2, equation 7}
    Q_1(\epsilon,\alpha)=\mathcal{O}\left(\frac{1}{\epsilon^{\alpha/(\alpha+1)}} \min \left\{\log(n),\log(1/\epsilon) \right\} \right).
\end{equation}
By using the argument in the $\alpha=0$ case, it can be proved that when $n<\lceil 1/\epsilon^{1/(\alpha+1)}\rceil$, $Q_2(\epsilon,\alpha)=\mathcal{O}(\log(n)+(1/n)\epsilon^{(-1)/(\alpha+1)})$. Hence for all $\alpha\in \left(0,1\right]$, 
\begin{equation}
\label{theorem 2, equation 8}
    Q_2(\epsilon,\alpha)=\mathcal{O}\left(\left(1+\frac{1}{n}\frac{1}{\epsilon^{1/(\alpha+1)}}\right)\min\left\{\log(n),\log(1/\epsilon)\right\} \right).
\end{equation}
Combining (\ref{theorem 2, equation 1}), (\ref{theorem 2, equation 3}), (\ref{theorem 2, equation 6}), (\ref{theorem 2, equation 7}), and (\ref{theorem 2, equation 8}), and noting that $\xi=1/(b c)$, the complexity in expectation of Algorithm \ref{acceleration framework} is
\begin{equation*}
    \mathbb{E}C_{\text{comp}}=\begin{cases}
    \mathcal{O}\left(\frac{b}{\epsilon^{1/(\alpha+1)}}+n\min\left\{\log(n),\log(1/\epsilon)\right\} \right), \quad &\alpha\in\left(0,\min \left\{\hat{\alpha},1/10\right\}\right], \\
    \mathcal{O}\left(\frac{b}{\epsilon^{1/(\alpha+1)}}+\left(\frac{n}{b}\frac{1}{\epsilon^{\alpha/(\alpha+1)}}+n\right)\min\left\{\log(n),\log(1/\epsilon)\right\} \right), \quad &\alpha\in \left(\min\left\{\hat{\alpha},1/10\right\},1\right].
    \end{cases}
\end{equation*}
Note that the above complexity can include the $\mathcal{O}(n\log (1/\epsilon)+1/\epsilon)$ result for $\alpha=0$ as a special case. Theorem \ref{theorem, complexity} is thus proved. $\hfill\qedsymbol$

\section{Discussions Concerning the Analysis of the SIFAR Method}
\label{comments on SIFAR}

Some comments concerning the theoretical analysis of the SIFAR method are listed as follows.

First, Theorem $1$ of \cite{li2025sifar} is incorrect. In Lemma $3$ of \cite{li2025sifar}, an upper bound is given for the quantity $\mathbb{E}\left\{[\eta_{t-1}/(p_{t-1}\theta_{t-1}^2)][f(w_t)-f(x^*)]\right\}$ for any $t>t_1$. Since $t_1$ is a random variable (r.v.) that may take any positive integer value in the probability space, $t$ should also be a r.v. to satisfy the inequality $t>t_1$, hence the coefficient $\eta_{t-1}/(p_{t-1}\theta_{t-1}^2)$ is a r.v. since it depends on $t$ and $t_1$. In \cite{li2025sifar}, this coefficient is directly taken out of the expectation $\mathbb{E}$ and both sides of the inequality are multiplied by $(p_{t-1}\theta_{t-1}^2)/\eta_{t-1}$ to obtain an upper bound on the concerned quantity $\mathbb{E}(f(w_t)-f(x^*))$. However, this reasoning is illegal and the upper bounding of $\mathbb{E}(f(w_t)-f(x^*))$ is thus elusive. Hence, we believe Theorem $1$ of \cite{li2025sifar} is incorrect and the difficulty there is critical.

Consider two correlated r.v.'s $X_1$ and $X_2$, the estimation of terms of the form $\mathbb{E}(X_1 X_2)$\footnote{Sometimes such terms take the form of the expectation of a sum of products of correlated r.v.'s, i.e., $\mathbb{E}\sum_{t=0}^{T} X_t^{(1)} X_t^{(2)}$, where $X_t^{(1)}$ and $X_t^{(2)}$ are correlated r.v.'s for $0\leq t\leq T$. But we take the simple $T=0$ case for the ease of discussion.} is important in the analysis of randomized methods. In some cases, we know how to estimate $\mathbb{E}(X_2)$ or $\mathbb{E}(X_2)$ is the concerned quantity, and we want to apply the known inequality or obtain $\mathbb{E}(X_2)$ alone. For example, such terms arise when the step size involves the current and/or some previous (incremental) gradients, e.g., the step size is determined by the Barzilai-Borwein-type formulae \cite{barzilai1988two} or the AdaGrad-type formulae \cite{duchi2011adaptive}. Then whether one can successfully estimate such terms determines the success or failure of the convergence analysis \cite{ward2020adagrad}. Moreover, when the estimation is feasible, the way of doing the estimates critically affects the ultimate theoretical complexity \cite{kavis2022adaptive}. Hence, we believe the problem that Theorem $1$ of \cite{li2025sifar} faced is beyond a trivial mis-reasoning.

Second, Corollary $1$ of \cite{li2025sifar} is incorrect. Due to the fact that $t_1$ is a r.v., the iterations needed by the SIFAR algorithm to achieve an $\epsilon-$accurate solution is also a r.v. To the best of our knowledge, this is in contrast to other existing single-loop VR methods, where the required iterations is deterministic given the target accuracy $\epsilon$. Hence the analysis of the complexity should be conducted differently. For a single-loop VR method with probabilistic update of the checkpoint, denote the computational cost of the $t$th iteration as $C_t$ (note that $C_t$ is a r.v.), and denote the iterations needed to reach an $\epsilon$-accurate solution as $T(\epsilon)$. Then the total cost in expectation is $\mathbb{E}\sum_{t=0}^{T(\epsilon)} C_t$. When $T(\epsilon)$ is deterministic, we have $\mathbb{E}\sum_{t=0}^{T(\epsilon)}C_t=\sum_{t=0}^{T(\epsilon)} \mathbb{E}(C_t)$. However, when $T(\epsilon)$ is a r.v., certain propositions, e.g., the Wald's theorem or its generalized versions \cite{ash2000probability,durrett2019probability}, should be applied in such case (if they are applicable at all under the design and characteristics of the SIFAR algorithm).

Hence, despite the appealing and impressive novelty and innovation of \cite{li2025sifar}, we believe the theoretical results under the general convex setting of \cite{li2025sifar} are invalid.

\bibliography{sn-bibliography}% common bib file
%% if required, the content of .bbl file can be included here once bbl is generated
%%\input sn-article.bbl

\section*{Statements and Declarations}

\bmhead{Funding}

This work is supported by National Key R\&D Program of China (2021YFA1000403) and the National Natural Science Foundation of China (Nos. 12431012, U23B2012).

\bmhead{Conflict of interest}

The authors have no relevant financial or non-financial interests to disclose.

\end{document}